\newfont{\Bbb}{msbm10 scaled\magstephalf}
\documentclass[12pt,reqno]{amsart}
\usepackage{amsfonts}
\usepackage{amsmath, amsthm, amscd, amsfonts}
\usepackage{amssymb}
\usepackage{amsmath}
\allowdisplaybreaks[4]
 \newtheorem{thm}{Theorem}[section]
 \newtheorem{cor}[thm]{Corollary}
 \newtheorem{lem}[thm]{Lemma}
 \newtheorem{prop}[thm]{Proposition}
 \theoremstyle{definition}
 \newtheorem{defn}[thm]{Definition}
\theoremstyle{remark}
 \newtheorem{rem}[thm]{Remark}
 
 \numberwithin{equation}{section}

\newcommand{\hh}{\mathbb{H}}

\newcommand{\pf}{\begin{proof}}
\newcommand{\zb}{\end{proof}}

\newcommand{\mc}{\mathbb{C}_I}
\newcommand{\ms}{\mathbb{S}}
\newcommand{\la}{\langle}
\newcommand{\ra}{\rangle}
 \newcommand{\mf}{\mathcal{F}^2(\mathbb{H})}
\begin{document}
\title[Weighted composition operator ]{ Weighted composition operator on quaternionic Fock space}

\author[P. Lian, Y. X. Liang]{Pan Lian and Yu-Xia Liang}
 \address{\newline College  of Mathematical Science,
Tianjin  Normal University, Tianjin 300387, P.R. China} \email{pan.lian@outlook.com}
\address{College of Mathematical Science,
Tianjin  Normal University, Tianjin 300387, P.R. China.} \email{liangyx1986@126.com}




\begin{abstract}

In this paper, we study the weighted composition operator on the Fock space $\mf$ of slice regular functions. First, we  characterize the boundedness and compactness of the weighted composition operator. Subsequently, we describe all the isometric composition operators. Finally, we introduce a kind of (right)-anti-complex-linear weighted composition operator on $\mf$  and obtain some concrete forms such that this  (right)-anti-linear weighted composition operator is a (right)-conjugation. Specially, we present equivalent conditions ensuring   weighted composition operators which are conjugate $\mathcal{C}_{a,b,c}-$commuting  or complex $\mathcal{C}_{a,b,c}-$ symmetric on $\mf$, which generalized the classical results on $\mathcal{F}^2(\mathbb{C}).$ At last part of the paper, we exhibit the closed expression for the kernel function of $\mf.$
\end{abstract}
\maketitle
\section{Introduction}

Complex symmetric operators have become particularly important both in theoretical and many application aspects in recent years(see e.g. \cite{GP}). It's well-known that all scalar (Jordan) models in operator theory are complex symmetric with respect to a well-chosen orthonormal basis (see, e.g. \cite{Ni}).  Technically speaking,  a bounded linear operator $T$ on a complex Hilbert space $\mathcal{H}$ is called $\mathcal{C}$-symmetric if \[\mathcal{C}T=T^*\mathcal{C}\] where $\mathcal{C}$ is an  anti-linear, isometric involution $\mathcal{C}$. This symmetry has strong effects on the spectral picture of $T$. Especially, the generalized eigenspaces $Ker(T-\lambda)^p$ and $Ker (T^*-\overline{\lambda})^p$ are antilinearly isomerically isomorphic via $\mathcal{C}.$  The general study of complex symmetric operators on Hilbert spaces was initiated by Garcia, Putinar and Wogen  (\cite{GP1,GP2,GP3,GW}), which presented naturally many new classes of complex symmetric operators. Thereafter,   quite a bit of attention has been paid to this topic and  many examples of $\mathcal{C}$-symmetric operators were found, such as Volterra integration operators, normal operators and certain Toeplitz operators. Recently, Hai and Khoi presented the weighted composition operators which are complex symmetry on the Fock space  of entire functions in \cite{HK1,HK2, Zhu1}.  In particular, they introduced a new class of  conjugations and found the criteria for bounded weighted composition operators which are complex symmetric with respect to their conjugations.

 In the last ten years, the theory of slice regular functions has equally got a lot of interests. It has  been developed systematically, and  found wide range of applications. It is  especially promising for the study of quaternionic quantum mechanics. Indeed, a new functional calculus  for slice functions was established  which could be considered as  the mathematically  framework of the  quaternionic quantum mechanics. In this paper, we fix our attention on the  weighted composition operators acting on Fock space in the slice setting, which are the set of  regular functions defined on the  quaternions. To our best of knowledge, the complex symmetry analogues  in the hyper-complex (quaternion or Clifford algebra) setting has not been studied, although it is  important for the study of the dynamic properties of the quaternionic composition operators.  

The aim of the present paper is  therefore to explore  and construct complex symmetric operators on the slice holomorphic Fock space. There exist different ways to define the slice regular composition in the literature, which can induce several definitions for weighted composition operators. In this paper, we study the weighted composition operator $W_{f, \varphi}$ for the slice regular functions defined by the slice regular product (see Section 2 ) $$(W_{f,\varphi}g)_{I}(z)=f_{I}\star(g_{I}\circ \varphi_{I})(z)$$ where $\varphi(\mc)$ satisfies the slice condition: $\varphi(\mc)\subset \mc$ for some $I\in \mathbb{S}$. The slice regular product makes the theory of composition operator more complicated than  the complex cases especially when one tries to  determine all the isometric weighted composition operators. Fortunately, by the slice regular product, we could characterize the boundedness, compactness, all the isometric weighted composition as nicely as the complex cases.

 Following the work in  \cite{HK1}, a complex  conjugation associated with the slice $\mc$ on the slice Fock space was introduced.  We offer sufficient and necessary condition for weighted composition operators which are complex symmetric with respect to our general conjugations.

 The paper consists of 6 sections and its outline is as follows.  In Section 2 we set up the basic facts concerning the regular functions, slice hyperholomorphic quaternionic  Fock space and weighted composition operator. In Section 3 we   investigate equivalent characterizations for the boundedness and compactness of the weighted composition operator acting on the slice Fock space. Section 4 is devoted to determine all the isometric weighted composition operators.  In Section 5, we define a general weighted composition   right-conjugation $\mathcal{C}_{a,b,c,d}$ on $\mf$. Then we explore the weighted composition operators which commutes with our special conjugation $\mathcal{C}_{a,b,c}$. Finally we derive a  theorem concerning  bounded weighted composition operators which are complex symmetric with respect to our conjugation $\mathcal{C}_{a,b,c}$ on $\mf.$

\section{preliminaries}
\subsection{Quaternions and slice regular function} The symbol $\mathbb{H}$  denotes the noncommutative, associative, real algebra of quaternions with standard basis $\{1, i, j, k \},$ subject to the multiplication rules $$i^2=j^2=k^2=ijk=-1.$$  That is to say $\hh$ is the set of the quaternions $$q=x_0+x_1i+x_2j+x_3k=Re (q)+ Im (q),$$ with $Re (q)=x_0$ and $ Im(q)=x_1i+x_2j+x_3k$, where $x_j\in \mathbb{R}$ for $j=1,2,3$.  The Euclidean norm  of a quaternion $q$ is given by
$$|q|=\sqrt{q\overline{q}}=\sqrt{x_0^2+x_1^2+x_2^2+x_3^2},$$ where $\overline{q}=Re(q)-Im(q)=x_0-(x_1i+x_2j+x_3k)$ representing the conjugate of $q$.  By the symbol $\mathbb{S}$ we denote the two-dimensional unit sphere of purely imaginary quaternions, i.e. $$\mathbb{S}=\{q=x_1i+x_2j+x_3k:\;x_1^2+x_2^2+x_3^2=1\}.$$ That is, $I^2=-1$ for $I\in \mathbb{S}.$  For any fixed $I\in \mathbb{S}$ we define $$\mathbb{C}_I:=\{x+Iy:\;x,y\in \mathbb{R}\},$$ which can be identified with a complex plane.  In the sequel, an element in the complex plane $\mc=\mathbb{R}+I\mathbb{R}$ is denoted by $x+Iy.$ Moreover, it holds that
$$\mathbb{H}=\bigcup\limits_{I \in \mathbb{S}}\mathbb{C}_I.$$  Interestingly, the real axis belongs to $\mathbb{C}_I$ for every $I\in \mathbb{S}$ and thus a real number can be associated with any imaginary unit $I$. However  any non real quaternion $q$ is uniquely associated to the element $I_q\in \mathbb{S}$ given by \[I_q:=(ix_i+jx_2+kx_3)/|ix_1+jx_2+kx_3|,\] and then $q$ belongs to the complex pane $C_{I_q}$.

Now, we are in a position to give the key concepts of this paper.
\begin{defn} \cite[Definition 2.1.1]{CSS2}\label{regularity} Let $U$ be an open set in $\hh$ and a function $f:\; U\rightarrow \hh$ be real differentiable. The function $f$ is  called \emph{slice regular} or \emph{slice hyperholomorphic} if,  for every $I\in \ms,$ its restriction  \[f_I(x + Iy) = f(x + Iy)\] is holomorphic, i.e. it has continuous partial derivatives and satisfies
$$\overline{\partial_I}f(x+yI):=\frac{1}{2}\left(\frac{\partial}{\partial x}+I\frac{\partial}{\partial y}\right)f_I(x+yI)=0$$ for all  $x+yI\in U\cap \mc$. The class of all slice regular functions on $U$ is denoted by $\mathcal{R}(U).$  Also $\mathcal{R}(U)$ is a right linear space on $\hh.$   \end{defn} Furthermore,  a function  which is  slice regular on $\hh$ will be called \emph{entire slice regular} or \emph{entire slice hyperholomorphic}, which is collected as $\mathcal{R}(\hh).$

Let $I, J \in \mathbb{S}$ be such that $I$ and $J$ are orthogonal, so that $I, J, IJ $ is an orthogonal basis of $\hh$ and write the restriction $f_I$  as the function \[f =f_0+If_1+Jf_2+IJf_3\]  on the complex plane $\mc$. It can also be written as $f_I = F + GJ $ where $f_0 + If_1 = F$, and $f_2 + If_3 = G.$  Hence we have the following splitting lemma, which relates slice regularity with classical holomorphy.
\begin{lem}\cite[Lemma 2.1.4]{CSS2} $($\textbf{Splitting Lemma}$)$ If $f$ is a slice regular function on the domain $U$, then for every $I, J\in \mathbb{S},$ with $I\perp J$, there are two holomorphic functions
$F,\;G:\;U_I=U\cap \mathbb{C}_I\rightarrow \mathbb{C}_I$ such that   $$f_I(z)=F(z)+G(z) J\;\;\mbox{for any} \;z=x+yI\in U_I.$$ \end{lem}

In the sequel, we always denote the symbols $z=x+yI$ and $w=u+vI$  on $\mc$ unless otherwise stated.

 Slice regular functions possess good properties on specific open sets that are called axially symmetric slice domains.

\begin{defn}\cite[Definition 2.2.1]{CSS2} Let $U\subset \hh$ be a domain.\\(1)\; $U$ is called a \emph{slice domain} (or $s$-domain for short) if it intersects the real axis and if, for any $I\in \ms,$ $U_I:=\mc\cap U$ is a domain in $\mc.$ \\ (2)\; $U$ is  \emph{axially symmetric}    if for every  $x+yI\in U$ with $x,y\in \mathbb{R}$ and $I\in \ms,$ all the elements $x+y\ms=\{x+yJ: \;J\in \ms\}$ is contained in $U.$ \end{defn}

The representation formula  of a slice regular function on an axially symmetric domain  allows  to recover  all its values   from its values on a single slice $\mc.$

\begin{prop}\label{Prop RF}\cite[Theorem 2.2.4]{CSS2}$($\textbf{Representation Formula}$)$\; Let $f$ be a slice regular function on an axially symmetric $s$-domain $U\subset \hh.$ Let $J\in \mathbb{S}$ and let $x\pm yJ\in U\cap \mathbb{C}_J,$  then the following equality holds for all $q=x+yI\in U,$
$$f(x+yI)=\frac{1}{2}\left[(1+IJ)f(x-yJ)+(1-IJ)f(x+yJ)\right].$$ \end{prop}

For more on the entire slice regular functions, we refer  to the excellent books \cite{CSS1,CSS2} and the references therein.


\subsection{Fock space in the slice setting}
In the following, we recall  the slice hyperholomorphic quaternionic  Fock space introduced in \cite{ACSS}.
\begin{defn} \cite[Definition 3.6]{ACSS} Let $I$ be any elements in $\mathbb{S}$ and $p|_{\mc}=z$,  consider the set
$$ \mathcal{F}^2(\mathbb{H})=\{f\in \mathcal{R}(\mathbb{H}):\;\int_{\mathbb{C}_I} e^{-|z|^2}|f_I(z)|^2 d\sigma(x,y)<\infty\}$$ where $d\sigma(x,y):=\frac{1}{\pi} dxdy.$ We  call $\mathcal{F}^2(\mathbb{H})$ (slice hyperholomorphic or slice regular) quaternionic Fock space.\end{defn}
 Here $\mathcal{F}^2(\mathbb{H})$ is endowed with the inner product $$\langle f, g \rangle =\int_{\mathbb{C}_I} e^{-|z|^2}\overline{g_I(z)}f_I(z) d\sigma(x,y),$$ which has been prove the definition of Fock space does not depend on the imaginary unit $I\in \mathbb{S},$ (see, e.g. \cite[Proposition 3.8]{Le}).   The norm induced by the inner product is
\begin{align} \|f\|=\left(\int_{\mathbb{C}_I} e^{-|z|^2} |f_I(z)|^2 d\sigma(x,y)\right)^{1/2}. \label{norm}  \end{align} It has been shown that $\mathcal{F}^2(\hh)$ contains the monomials $p^n$ ($n\in \mathbb{N}$), which form an orthogonal basis and $\la p^n,p^n\ra=n!$. Moreover, a function \[f(p)=\sum_{m=0}^\infty p^m a_m\in \mathcal{F}^2(\mathbb{H})\] if and only if $\sum_{m=0}^\infty |a_m|^2m!<\infty.$

Given a variable $p\in \mathbb{H}$ and a parameter $q\in \mathbb{H}$, we set
\begin{eqnarray}e_{*} ^{pq}=\sum_{n=0}^{+\infty} \frac{(pq)^{*n}}{n!}=\sum_{n=0}^{+\infty} \frac{(pq)^{\star n}}{n!}=\sum_{n=0}^{+\infty} \frac{p^nq^n}{n!},\label{closed}\end{eqnarray} it is immediate that  $e_{*}^{pq}$ is a function left slice regular in $p$ and right regular in $q$. The reproducing kernel of  $\mathcal{F}^2(\mathbb{H})$ is given by $K_q(p):=e_{*}^{p\overline{q}}$, and it holds that  $\langle f, K_q\rangle=f(q)$ for any $f\in \mathcal{F}^2(\mathbb{H})$ (see, e.g. \cite[Theorem 3.10]{ACSS}). Besides,
\begin{align}\|K_q\|^2=\la K_q, K_q\ra=K_q(q)=e^{|q|^2}.\label{ker-norm} \end{align} Furthermore we  denote $k_q(p)=K_q(p)/\|K_q\|$, which is a unit-vector in $\mf.$ More generally, we obtain
$$f^{(m)}(u)=\int_{\mc}f_I(z)\overline{z}^m e_*^{\overline{z}u}e^{-|z|^2}d\sigma(x,y)=\la f, K_u^{[m]}\ra,$$ for $K_u^{[m]}(q)=q^mK_u(q),$ the $m-$th derivative evaluation kernel. Particularly, we  take $u=0$ and obtain $P_m(p)=p^m:=K_0^{[m]}(p),\;m\geq 0,$ which will be used in Section 6.

\subsection{Weighted composition operators} In this subsection, we introduce the weighted composition operator $W_{f,\varphi}$ defined on $\mf$.

  Since pointwise product of functions does not preserve slice regularity,  a new multiplication operation, the $\star$-product, was introduced in order to maintain the regularity in  \cite{CSS1,CSS2}.  In the special case of power series, the regular product (or  $\star$-product) is given below. Let $U$ be  a ball with center at a real point, $f(q)=\sum_{n=0}^\infty q^n a_n,$ $a_n\in \mathbb{H}$ and $g(q)=\sum_{n=0}^\infty q^n b_n$ with $b_n\in \hh$,  the regular product of $f$ and $g$ is defined as
$$(f\star g)(q):=\sum_{n=0}^\infty q^n \left(\sum_{r=0}^n a_r b_{n-r}\right).$$ In this case, the notion $\star$-product coincides with the classical notion of product of series with coefficients in a ring. It is easy to see the function $f\star g$ is   slice hyperholomorphic. The regular product  was further generalized to the case of regular functions defined  on axially symmetric $s$-domains.  Let $U\subset \hh$ be an axially symmetric s-domain and let $f,g:\; U\rightarrow \hh$ be slice regular functions. For any $I, J\in\ms$ with $I\perp J,$ the Splitting Lemma guarantees the existence of four holomorphic functions $F, G, H, K: \;U\cap \mc\rightarrow \mc$ such that for all $z=x+yI\in U\cap \mc,$
$$f_I(z)=F(z)+G(z)J,\;\;g_I(z)=H(z)+K(z)J.$$ Then $f_I\star g_I:\;U\cap \mc\rightarrow \mc$ is defined as
\begin{eqnarray}&&f_I\star g_I (z)\nonumber\\&=&[F(z)H(z)-G(z)\overline{K(\overline{z})}]+[F(z)K(z)+G(z)\overline{H(\overline{z})}]J.\;\;\;
\label{star} \end{eqnarray}
The new function $f_I\star g_I$ is  a holomorphic map and hence it admits an unique slice regular extension to $U$ defined by $ext(f_I\star g_I)(q)$. Hence,
\begin{defn}\label{defn regular-product} Let $U\subset \hh$ be an axially symmetric $s$-domain and let $f, g:\;U\rightarrow  \hh$ be slice regular. The function  $$(f\star g)(q)=ext(f_I\star g_I)(q)$$ defined as the extension of  \eqref{star} (using Proposition \ref{Prop RF}) is called the slice regular product of $f$ and $g$.  \end{defn}

 \begin{rem}By the definition of $\star$-product we obtain  \begin{eqnarray}J\star H(z) =\overline{ H(\overline{ z})}J.\label{starH}\end{eqnarray}
\end{rem}
\begin{rem} By formula (\ref{star}), the star product could be formally defined on any two quaternion valued function. We will use this notation in Section 3  to make the similarity  between the slice setting and the classical case clear.
\end{rem}

The $\star-$product can be described  by the pointwise product as  the following Proposition.
\begin{prop}\cite[Theorem 2.3.10]{CSS2} Let $U\subset \hh$ be an axially symmetric $s-$domain,
 $f, g: U\rightarrow \hh$ be slice hyperholomorphic functions. Then
\begin{eqnarray}(f\star g)(p)=f(p)g(f(p)^{-1}pf(p)),\label{star*}\end{eqnarray}for all $p\in U,$ $f(p)\neq 0,$ while $(f\star g)(p)=0$ when $p\in U,  f(p)=0.$ \end{prop}
 \begin{rem}\label{rem R} Let $p\in \mathbb{R},$ then $(f\star g)(p)f(p)g(f(p)^{-1}pf(p))=f(p)g(p),$ which implies that the regular product of two slice function reduces to the usual product when the variable takes real values.  \end{rem}

  Let $\varphi:\;\hh\rightarrow \hh$ be a slice hyperholomorphic map such that $\varphi(\mc)\subset \mc$ for some $I\in \ms.$ The composition operator $C_\varphi$ on $\mf$ induced by $\varphi$ is defined  by
\[(C_\varphi f)_I(z)= (f_I \circ \varphi_I)(z)=F\circ\varphi_I (z)+G\circ \varphi_I(z)J \] for all $f\in \mf$. By the representation formula (Proposition \ref{Prop RF}), we  get the extension $C_\varphi h$ to the whole $\hh.$ The systematical research of composition operators on  holomorphic functions in one complex variables  is a fairly active field, see the two books \cite{CM} by Cowen and MacCluer, and \cite{Sh} by Shapiro. As regards to  recent topic about composition operators on slice regular spaces of hyperholomorphic functions, we refer to \cite{KM,VCG} and their references therein.

 Combining the definition of slice regular product in Definition \ref{defn regular-product} and the composition operator, we define the weighted composition operators  $W_{f,\varphi}:\;\mf\rightarrow \mf$ with $\varphi:\;\mc\rightarrow \mc$  for some $I\in \ms$ as
$$(W_{f,\varphi} h)_I(z)= f_I(z)\star (h_I\circ \varphi_I)(z)$$ for all $f\in \mf$.   The extension $W_{f,\varphi} h$  on  $\hh$ is obtained by the representation formula.  Here we also note that
$$[W_{f,\varphi}(g a+ hb)]_I(z)=(W_{f,\varphi} g)_I(z) a+(W_{f,\varphi} h)_I(z)b$$ for any $g, h\in \mf$ and $a, b \in\hh,$ which means $W_{f,\varphi}$ are right-linear on $\mf.$

\section{Boundedness and  compactness}

 Let $f$ and $\varphi$ be two slice regular functions on $\mf$  such that $f$ is not identically zero and $\varphi(\mc)\subset \mc$ for some $I\in\ms$. In this section we will systematically investigate  the boundedness and compactness of weighted composition operators $W_{f,\varphi}$ on $\mf.$  In the following,  the adjoint of the weighted composition operator $W_{f,\varphi}$ is denoted  by $W_{f,\varphi}^*$.  We will always choose
 $J\in\ms$ with $I\perp J$ and express \begin{eqnarray} f_I(w)=F(w)+G(w)J\label{fI}\end{eqnarray} with two holomorphic functions $F, G: \mc\rightarrow \mc.$

 For the reproducing kernel $K_p\in \mf,$ we have
\begin{align*} \la W_{f,\varphi}^* K_p, K_q\ra=\la K_p, W_{f,\varphi} K_q\ra =\overline{\la W_{f,\varphi}K_q,K_p\ra},\end{align*} which entails that
\begin{align}W_{f,\varphi}^* K_p(q)= \overline{W_{f,\varphi}K_q(p)}.\label{ker}\end{align}

Now, we give a sufficient and necessary characterizations for bounded weighted composition operator on $\mf.$

 \begin{thm} \label{thm bou}Let $f$ and $\varphi$ be two slice regular functions on $\mathbb{H},$ such that $f$ is not identically zero and $\varphi(\mathbb{C}_I) \subset \mathbb{C}_I$ for some $I\in \mathbb{S}$.  Then the operator $W_{f,\varphi}:\;\mf\rightarrow \mf$ is bounded if and only if $f\in \mathcal{F}^2(\mathbb{H}),$ $\varphi(p)=p\lambda+\varphi(0)$ with $|\lambda|\leq 1$, $\lambda\in \mc$ and
 \begin{eqnarray} \hat{M}(f,\varphi):&=&\sup\limits_{w\in \mc}\left(|F(w)|^2 e^{|\varphi(w)|^2-|w|^2}+ |G(w)|^2e^{|\varphi(\overline{w})|^2-|w|^2}\right)
 \nonumber\\&=&\sup\limits_{w\in \mc}\left| f_{I}(w)\star e^{\frac{|\varphi(w)|^2-|w|^2}{2}}\right|^{2} <\infty.\label{Mf1}\;\;\;\;\;\end{eqnarray}
  \end{thm}
\pf \emph{Necessity}. Suppose  $W_{f,\varphi}$ is bounded on $\mathcal{F}^2(\mathbb{H}),$ then  \[f=f\star 1=W_{f,\varphi} 1\in \mathcal{F}^2(\mathbb{H}).\] On the other hand, we know \begin{align} \|W_{f,\varphi}\|^2=\|W_{f,\varphi}^*\|^2\geq \frac{\|W_{f,\varphi}^* K_p\|^2}{\|K_p\|^2}.\label{Wf-norm} \end{align}
With the help of \eqref{ker}, the norm $\|W_{f,\varphi}^* K_p\|^2$  becomes
\begin{eqnarray} &&\|W_{f,\varphi}^* K_p\|^2\nonumber\\&=&\|W_{f,\varphi}^* K_p(q)\|^2=\|\overline{W_{f,\varphi} K_q(p)}\|^2\nonumber\\&=&
\|W_{f,\varphi} K_q(p)\|^2=\|f(p)\star K_q\circ \varphi(p)\|^2\nonumber\\&=&\int_{\mc}\left|f(p)\star \sum_{n=0}^\infty \frac{(\varphi(p))^{\star n}\overline{q}^n}{n!}\right|^2 e^{-|q|^2} d\sigma(x,y) \nonumber\\&= &\int_{\mc}\left|f(p)\star \sum_{n=0}^\infty \frac{(\varphi(p))^{\star n}\overline{z}^n}{n!}\right|^2 e^{-|z|^2} d\sigma(x,y)\\&\geq&
\int_{\mc}\left|f_I(w)\star \sum_{n=0}^\infty \frac{(\varphi_I(w))^{\star n}\overline{z}^n}{n!}\right|^2 e^{-|z|^2} d\sigma(x,y),\label{norm1} \end{eqnarray} where $q|_{\mc}= z=x+yI \in\mc$  and $p|_{\mc}=w=u+vI\in\mc$.

Now we express the modulus of the regular product \[f(p)\star \sum_{n=0}^\infty \frac{(\varphi_I(p))^{\star n}\overline{z}^n}{n!}\] explicitly.   Since $\varphi:\;\mc \rightarrow \mc,$ it follows that $$\varphi(p)^{{\star}n}=ext(\varphi_I(w)^{\star n})=ext (\varphi_I(w)^n).$$ Hence, on $\mc$,  by the operation rule \eqref{star} it turns out
 \begin{eqnarray*} &&f_I(w)\star \sum_{n=0}^\infty \frac{(\varphi_I(w))^{\star n}\overline{z}^n}{n!} \nonumber\\&=&(F(w)+G(w)J)\star \sum_{n=0}^\infty\frac{(\varphi_I(w))^{ n}\overline{z}^n}{n!} \nonumber\\&=&
 F(w) \cdot \sum_{n=0}^\infty\frac{(\varphi_I(w))^{ n}\overline{z}^n}{n!}+ G(w)\cdot \sum_{n=0}^\infty \frac{\overline{(\varphi_I(\overline{w}))^{ n}\overline{z}^n}}{n!} \cdot J,   \end{eqnarray*} which ensures
 \begin{eqnarray} &&\left|f(p)\star \sum_{n=0}^\infty \frac{(\varphi_I(p))^{\star n}\overline{z}^n}{n!}\right|^2\nonumber\\& =&\left|  F(w) \cdot \sum_{n=0}^\infty\frac{(\varphi_I(w))^{ n}\overline{z}^n}{n!}\right|^2+\left| G(w)\cdot \sum_{n=0}^\infty \frac{\overline{(\varphi_I(\overline{w}))^{ n}\overline{z}^n}}{n!} \right|^2\nonumber\\&=&\left|  F(w) \right|^2\left|  \sum_{n=0}^\infty\frac{(\varphi_I(w))^{ n}\overline{z}^n}{n!}\right|^2+\left| G(w)\right|^2\left| \sum_{n=0}^\infty \frac{ (\varphi_I(\overline{w}))^{ n}\overline{z}^n}{n!} \right|^2 .\nonumber\end{eqnarray}

 Therefore, the display \eqref{norm1} becomes
 \begin{eqnarray} &&\|W_{f,\varphi}^* K_p\|^2\nonumber\\&=&\int_{\mc} \left|  F(w) \right|^2\left|  \sum_{n=0}^\infty\frac{(\varphi_I(w))^{ n}\overline{z}^n}{n!}\right|^2e^{-|z|^2} d\sigma(x,y) \nonumber\\&&+\int_{\mc} \left| G(w)\right|^2\left| \sum_{n=0}^\infty \frac{ (\varphi_I(\overline{w}))^{ n}\overline{z}^n}{n!} \right|^2  e^{-|z|^2} d\sigma(x,y) \nonumber\\&=&   \left|  F(w) \right|^2 \sum_{n=0}^\infty \frac{|\varphi_I(w)|^{2n} n!}{(n!)^2}   +  \left| G(w)\right|^2 \sum_{n=0}^\infty \frac{ |\varphi_I(\overline{w})|^{2 n} n!}{(n!)^2}\nonumber\\&=& |F(w)|^2 e^{|\varphi_I(w)|^2}+|G(w)|^2 e^{|\varphi_I(\overline{w})|^2}.\label{norm*}  \end{eqnarray} Employing \eqref{ker-norm}, it yields that
 \begin{eqnarray*}\infty&>& \frac{\|W_{f,\varphi}^* K_p\|^2}{\|K_p\|^2}\\&=&\frac{|F(w)|^2 e^{|\varphi_I(w)|^2}+|G(w)|^2 e^{|\varphi_I(\overline{w})|^2}}{e^{|w|^2}}\nonumber\\&=& |F(w)|^2 e^{|\varphi_I(w)|^2-|w|^2}+|G(w)|^2 e^{|\varphi_I(\overline{w})|^2-|w|^2}.   \end{eqnarray*}   
Therefore we  obtain \begin{align}\sup\limits_{w\in \mc} |F(w)|^2 e^{|\varphi_I(w)|^2-|w|^2}+|G(\overline{w})|^2 e^{|\varphi_I( w)|^2- |w|^2}<\infty.\label{ww1} \end{align}
  The above inequality contains the following two inequalities,
  \begin{eqnarray}&& \sup\limits_{w\in \mc} |F(w)|^2 e^{|\varphi_I(w)|^2-|w|^2}<+\infty,\label{F1} \\&& \sup\limits_{w\in \mc}  |G(w)|^2e^{|\varphi_I(\overline{w})|^2-|w|^2} <+\infty.\label{F2}\end{eqnarray}
   Any one of the above inequalities together with
  \cite[Proposition 2.1]{Le}   imply that \[\varphi_I(w)=\varphi_I(0)+w \lambda\]
 with some $|\lambda|\leq 1$ for $w,\;\lambda\in \mc.$ By the representation formula (Proposition  \ref{Prop RF}), we extend $\varphi_I$ into the whole  $\hh$ \[\varphi(p)=\varphi(0)+p\lambda\] with $\lambda\in \mc$ and $|\lambda|\leq 1$. Indeed,  for $p=x+yI$ and $q=x+yJ,$  we have \begin{eqnarray*} \varphi(p)&=&\frac{1}{2}(1-IJ)\varphi(q)+\frac{1}{2}(1+IJ)\varphi(\overline{q}) \nonumber\\&=& \frac{1}{2}(1-IJ)(\varphi(0)+q\lambda)+\frac{1}{2}(1+IJ)(\varphi(0)+\overline{q}\lambda) \nonumber\\&=& \varphi(0)+\frac{1}{2}[(q+\overline{q})+IJ(\overline{q}-q)\lambda]\nonumber\\&=& \varphi(0)+\frac{1}{2}[2x+IJ(-2Jy)\lambda]\nonumber
 \\&=&\varphi(0)+(x+Iy)\lambda\nonumber\\&=&\varphi(0)+p\lambda.
 \end{eqnarray*}

\emph{Sufficiency}. We  show  the operator $W_{f,\varphi}: \mf\rightarrow \mf$ is bounded. Firstly, we suppose that $\lambda\neq 0.$ For any $h\in \hh,$ it follows that  \begin{eqnarray} &&\|W_{f,\varphi} h\|^2\nonumber\\&=&\int_{\mc}|[f\star (h\circ \varphi)]_I(p)|^2 e^{-|p|^2} d\sigma(x,y)\nonumber\\&=&\int_{\mc}|f_I\star (h\circ \varphi)_I(z)|^2 e^{-|z|^2} d\sigma(x,y).\label{h-norm}\end{eqnarray}
Denoting $f_I(z)= F(z)+G(z)J$ and $(h\circ \varphi)_I=H(z)+K(z)J,$  the display \eqref{star} becomes
 \begin{eqnarray*}&& f_I(z)\star (h\circ \varphi)_I(z)\nonumber\\&=&(F(z)H(z)-G(z)\overline{K(\overline{z})})+
 (F(z)K(z)+G(z)\overline{H(\overline{z})})J .\end{eqnarray*} Thus
 \begin{eqnarray}&& |f_I(z)\star (h\circ \varphi)_I(z)|^2\nonumber\\&=&|F(z)H(z)-G(z)\overline{K(\overline{z})}|^2+|F(z)K(z)+G(z)\overline{H(\overline{z})}|^2 \nonumber\\&\leq & 2(|F(z)H(z)|^2+|G(z)\overline{K(\overline{z})}|^2+|F(z)K(z)|^2+|G(z)\overline{H(\overline{z})}|^2 )\nonumber\\&=&2 |F(z)|^2(|H(z)|^2+|K(z)|^2)+2|G(z)|^2(|H(\overline{z})|^2+|K(\overline{z})|^2)
 \nonumber\\&=&2|F(z)|^2 |(h\circ \varphi)_I(z)|^2+2|G(z)|^2 |(h\circ \varphi)_I(\overline{z})|^2.\label{II1} \end{eqnarray}
 Putting \eqref{II1} into \eqref{h-norm}, we obtain
  \begin{eqnarray} &&\|W_{f,\varphi} h\|^2  \nonumber\\&\leq&  2\int_{\mc} |F(z)|^2 |(h\circ \varphi)_I(z)|^2 e^{-|z|^2} d\sigma(x,y)\nonumber\\&&+2\int_{\mc} |G(z)|^2 |(h\circ \varphi)_I(\overline{z})|^2 e^{-|z|^2} d\sigma(x,y)\;\;\;\;\;~~\label{computation}\\&\leq 2&\sup\limits_{z\in \mc} |F(z)|^2 e^{|\varphi_I(z)|^2-|z|^2}\int_{\mc}|(h\circ \varphi)_I( z)|^2 e^{-|\varphi_I(z)|^2} d\sigma(x,y)\nonumber\\&&+2\sup\limits_{z\in \mc} |G(z)|^2 e^{|\varphi_I(\overline{z})|^2-|z|^2}\int_{\mc}|(h\circ \varphi)_I(\overline{z})|^2 e^{-|\varphi_I(\overline{z})|^2} d\sigma(x,y) \nonumber\\&=& 2\hat{M}(f,\varphi)\int_{\mc}|(h\circ \varphi)_I(z)|^2 e^{-|\varphi_I(z)|^2} d\sigma(x,y) \nonumber\\&=& 2\hat{M}(f,\varphi)|\lambda|^{-2}\|h\|^2, \label{Suff1}\end{eqnarray}  where the change of variable $w=\varphi(z)$ was used in the last line. Therefore,  the operator $W_{f,\varphi}$  is bounded on $\mf$ for $\lambda\neq 0$.

   For the case $\lambda=0,$ it holds that $\varphi(z)=\varphi(0).$ Taking any $h\in\mf,$ we have  $$W_{f,\varphi} h=f\star ( h\circ\varphi(0))=f\star \la h, K_{\varphi(0)}  \ra\in \mf$$  due to $f\in \mf.$ As a matter of fact, similar to \eqref{II1}, it yields that  \begin{eqnarray*}&&|f_I(z)\star \la h, K_{\varphi(0)}\ra|^2\nonumber\\&\leq&   2|F(z)|^2 |(\la h, K_{\varphi(0)}\ra)_I(z)|^2+2|G(z)|^2 |(\la h, K_{\varphi(0)}\ra)_I(\overline{z})|^2\nonumber\\&\leq & 2|f_I(z)|^2\|h\|^2\|K_{\varphi(0)}\|^2.\end{eqnarray*} It follows that
  \begin{eqnarray*} &&\|W_{f,\varphi} h\|^2\nonumber\\&=&\int_{\mc}\left|f_I(z)\star \la h, K_{\varphi(0)} \ra\right|^2 e^{-|z|^2} d\sigma(x,y)\nonumber\\&\leq&2 \|f\|^2\| K_{\varphi(0)} \|^2\|h\|^2. \end{eqnarray*}  In a summary,  the operator $W_{f,\varphi}:\;\mf\rightarrow \mf$ is bounded for any  $|\lambda|\leq 1$ and $\lambda\in \mc$, which ends the proof.
\zb

We
proceed to give the explicit form of $f$ under the condition $W_{f,\varphi}$ is bounded on $\mf$ with $\varphi(p)= pa+b$, $ |a|=1.$
\begin{prop} \label{prop lambda1}Let $f$ and $\varphi$ be two slice regular functions on $\mathbb{H},$ such that $f$ is not identically zero and $\varphi(\mathbb{C}_I) \subset \mathbb{C}_I$ for some $I\in \mathbb{S}$.  If the weighted composition operator $W_{f,\varphi}$ is bounded on $\mf,$ then \[\varphi(p)= pa+b\] with $|a|\leq 1.$ If $|a|=1,$ then
 \begin{eqnarray*}f_I(z)=F(0) e^{-a\overline{b}z}+G(0) e^{-\overline{a}b z}J=(F(0)+G(0)J)\star e^{-a\overline{b}z}
  \end{eqnarray*}
 with  at least one of $F(0)$ and $G(0)$  is not zero.
  \end{prop}
  \pf Suppose the operator $W_{f,\varphi}$ is bounded on $\mathcal{F}^2(\mathbb{H}),$  Theorem \ref{thm bou} ensures the display \eqref{Mf1} holds. Furthermore, the inequalities \eqref{F1} and \eqref{F2} are valid.

   Under the case $\varphi_I(w)=wa+b$ with $|a|=1,$ denoting $\beta=\overline{a}b,$ it yields that
  \begin{eqnarray*}&&|\varphi_I(w)|^2-|w|^2=\overline{\beta}w+\beta \overline{w}+|b|^2\\&&|\varphi_I(\overline{w})|^2-|w|^2=\overline{\beta w}+\beta w+|b|^2. \end{eqnarray*}  Furthermore, \eqref{F1} and \eqref{F2} entail that \[|F(w)e^{\overline{\beta}w}|^2\leq \hat{M}(f,\varphi)e^{-|b|^2}\;\;\mbox{and}\;\; |G(w)e^{ \beta w }|^2\leq \hat{M}(f,\varphi)e^{-|b|^2},\] $\mbox{for all } \;w\in \mc.$ Hence we  deduce that
  \begin{eqnarray} F(w)=F(0)e^{-\overline{\beta}w}\;\;\mbox{and}\;\;G(w)=G(0)e^{-\beta w}.  \label{FG} \end{eqnarray}It follows that
   \[ f_I(z)=F(0) e^{-a\overline{b}z}+G(0) e^{-\overline{a}b z}J.\] Since  $f$ is not identically zero, hence at least one of $F(0)$ and $G(0)$  is not zero, which completes the proof.
\zb
Theorem \ref{thm bou} and Proposition \ref{prop lambda1} can validate the following Proposition, which gave an interesting example for boundedness of $W_{f,\varphi}$ when $f, \varphi$ are specific slice functions.
\begin{prop} \label{prop spe}Let $\varphi(p)=pA+B$ and $f_I(z)=C_1e^{D_1 z}+C_2e^{D_2 z}J,$ where $A, B, C_i$ and $D_i$ are complex constants for $i=1, 2.$ Then the weighted composition operator $W_{f,\varphi}$ is bounded on $\mf$ if and only if

$(a)$ either $|A|<1$

$(b)$ or $|A|=1,$ $D_1+A\overline{B}=0$ and $D_2+\overline{A}B=0.$\end{prop}
In the sequel, we describe the equivalent conditions for compactness of  $W_{f,\varphi}$ on $\mf.$
\begin{thm} \label{thm cpt}Let $f$ and $\varphi$ be two slice regular functions on $\mathbb{H},$ such that $f$ is not identically zero and $\varphi(\mathbb{C}_I) \subset \mathbb{C}_I$ for some $I\in \mathbb{S}$.  Then $W_{f,\varphi}$ is compact on $\mf$ if and only if  $\varphi(p)=\varphi(0)+p\lambda$ with $|\lambda|<1$, $\lambda\in \mc$ and
 \begin{eqnarray} &&\lim\limits_{|w|\rightarrow \infty}\left(|F(w)|^2 e^{|\varphi(w)|^2-|w|^2}+ |G(\overline{w})|^2e^{|\varphi(w)|^2-|w|^2}\right)\nonumber\\&=&  \lim\limits_{|w|\rightarrow \infty}\left| f_{I}(w)\star e^{\frac{|\varphi(w)|^2-|w|^2}{2}}\right|^{2}=0.\label{Mf2}\end{eqnarray}
  \end{thm}
\pf \emph{Necessity}. Suppose the operator $W_{f,\varphi}$ is compact on $\mf,$ it must be bounded and \[\varphi(p)=\varphi(0)+p\lambda\] with $|\lambda|\leq 1,$ $\lambda\in \mc.$ For the case $|\lambda|=1,$ Proposition \ref{prop lambda1} implies \eqref{FG} is true   for $\beta=\overline{\lambda}\varphi(0).$  Then
\begin{eqnarray}&&|F(w)|^2 e^{|\varphi(w)|^2-|w|^2}+|G(\overline{w})|^2 e^{|\varphi(w)|^2-|w|^2}\nonumber\\&
=&(|F(0)|^2+|G(0)|^2)e^{|\varphi(0)|^2}\nonumber\\&=&|f_I(0)|^2e^{|\varphi(0)|^2},\end{eqnarray}
which means \eqref{Mf2} does not converge to zero.

Thus  we suppose $|\lambda|<1$ and go on to show the display \eqref{Mf2}. If  $W_{f,\varphi}$ is compact on $\mf,$  we know that the adjoint operator  $W_{f,\varphi}^*$ is also compact on $\mf.$
Since  \[k_p=\|K_p\|^{-1}K_p\rightarrow 0\] as $|p|\rightarrow \infty,$ we have \[\|K_p\|^{-2}\|W_{f,\varphi}^* K_p\|^2=e^{-|p|^2} \|W_{f,\varphi}^* K_p\|^2\rightarrow 0\] as $|p|\rightarrow \infty.$ Employing the computation in \eqref{Wf-norm} and \eqref{norm*}, it yields that
\begin{align*} |F(w)|^2 e^{|\varphi_I(w)|^2-|w|^2}+|G(w)|^2 e^{|\varphi_I(\overline{w})|^2-|w|^2}\rightarrow 0,\end{align*} as $|w|\rightarrow \infty.$ That means \eqref{Mf2} holds.

\emph{Sufficiency}.  Assume that  \[\varphi(p)=\varphi(0)+p\lambda\] with $|\lambda|<1$, $\lambda\in \mc$ and \eqref{Mf2} holds. For the case $\lambda=0,$ this implication is obvious. In fact, \[W_{f,\varphi} h=f\star h(\varphi(0)),\] which  illustrates   $W_{f,\varphi} $ has finite rank, thus it is compact.

 Now suppose that $\lambda\neq 0,$ we proceed to prove the weighted composition operator is compact on $\mf.$ Let $\{h_m\}_{m=1}^\infty$ be a bounded sequence in $\mf$ with $C:=\sup\limits_{m\in \mathbb{N}}\|h_m\|<+\infty$ and converge weakly to $0$ as $m\rightarrow \infty.$ Then we know the sequence $\{h_m\}$ converges to zero uniformly on compact subsets of $\hh.$ In the sequel, we show that $$\|W_{f,\varphi} h_m\|^2 \rightarrow 0$$ as $m\rightarrow \infty.$ By the similar calculation in \eqref{computation}, we have
\begin{eqnarray}&&\|W_{f,\varphi}h_m\|^2\\&=&\int_{\mc }|f_I(z)\star (h_m\circ\varphi)_I(z)|^2e^{-|z|^2}d\sigma(x,y)\nonumber\\&\leq&  2\int_{\mc} |F(z)|^2 |(h_m\circ \varphi)_I(z)|^2 e^{-|z|^2} d\sigma(x,y)\nonumber \\&&+  2\int_{\mc} |G(\overline{z})|^2 |(h_m\circ \varphi)_I(z)|^2 e^{-|z|^2} d\sigma(x,y).  \label{I1+I2} \end{eqnarray}
Denoting \begin{align} &&I_1:= \int_{\mc} |F(z)|^2 |(h_m\circ \varphi)_I(z)|^2 e^{-|z|^2} d\sigma(x,y);\label{I1}\\&&I_2:=\int_{\mc} |G(\overline{z})|^2 |(h_m\circ \varphi)_I(z)|^2 e^{-|z|^2} d\sigma(x,y).\label{I2}\end{align}
we only need to show $I_1\rightarrow 0$ and $I_2\rightarrow 0$ as $m\rightarrow \infty.$

For $w\in \mc,$ define the function
$$\hat{F}(w)=|\lambda|^{-2}|F(\varphi_I^{-1}(w))|^2 e^{|w|^2-|\varphi_I^{-1}(w)|^2},$$ and then \[\hat{F}(\varphi_I(w))=|\lambda|^{-2} |F(w)|^2 e^{|\varphi_I^{-1}(w)|^2-|w|^2}.\] Since $\lim\limits_{|w|\rightarrow \infty} |\varphi_I^{-1}(w)|=\infty,$ the first part of \eqref{Mf2} yields $\lim\limits_{|w|\rightarrow \infty}\hat{F}(w)=0,$ hence $\hat{F}$ is a bounded function on $\mc$ with $\|\hat{F}\|_\infty<\infty.$ For any $\epsilon>0,$ there exists a $R>0$ such that  \begin{eqnarray}\sup\limits_{|w|>R}\hat{F}(w)<\epsilon. \label{epsilon1} \end{eqnarray}
Alternatively, it's easy to check that  $$|F(w)|^2=|\lambda|^2 \hat{F}(\varphi_I(w)) e^{|w|^2-|\varphi_I(w)|^2}.$$ Putting it into \eqref{I1} and employing \eqref{epsilon1},  we deduce that
\begin{align*} I_1&=  \int_{\mc} |\lambda|^2 \hat{F}(\varphi_I(z)) e^{|z|^2-|\varphi_I(z)|^2} |(h_m)_I\circ \varphi_I(z)|^2 e^{-|z|^2} d\sigma(x,y)\nonumber\\&= \int_{\mc} \hat{F}(w)e^{-|w|^2}|(h_m)_I(w)|^2d\sigma(u,v)\nonumber\\&= \int_{|w|\leq R}  \hat{F}(w)e^{-|w|^2}|(h_m)_I(z)|^2d\sigma(u,v)\nonumber\\&\;\;+\int_{|w|>R}  \hat{F}(w)e^{-|w|^2}|(h_m)_I(w)|^2d\sigma(u,v)\nonumber\\&\leq
\|\hat{F}\|_\infty \int_{|w|\leq R}  |(h_m)_I(w)|^2d\sigma(u,v)\nonumber\\&\;\;+\sup\limits_{|w|>R}\hat{F}(w)\int_{|w|>R}  e^{-|w|^2}|(h_m)_I(w)|^2d\sigma(u,v)\nonumber\\&\leq C^2 \epsilon  ,\;\mbox{as} \; m\rightarrow \infty. \end{align*} Due to $\epsilon$ is arbitrary,  then $I_1\rightarrow 0$ as $m\rightarrow \infty.$  We  note that  in the second line, the variable substitution $w=\varphi_I(z)$ was used.

Analogously,  define $$\hat{G}(w)=|\lambda|^{-2}|G(\overline{\varphi_I^{-1}(w)})|^2 e^{|w|^2-|\varphi_I^{-1}(w)|^2},$$ and it also holds $\lim\limits_{|w|\rightarrow \infty}\hat{G}(w)=0.$ Moreover,
$$|G(\overline{w})|^2=|\lambda|^2\hat{G}(\varphi_I(w))e^{|w|^2-|\varphi_I(w)|^2}. $$ Now
\begin{eqnarray*}I_2:&=&\int_{\mc} |\lambda|^2\hat{G}(\varphi_I(z))e^{|z|^2-|\varphi_I(z)|^2}| (h_m\circ \varphi)_I(z)|^2 e^{-|z|^2} d\sigma(x,y)\nonumber\\&=&\int_{\mc} |\lambda|^2\hat{G}(\varphi_I(z))e^{-|\varphi_I(z)|^2}| (h_m\circ \varphi)_I(z)|^2 d\sigma(x,y)\nonumber\\&=& \int_{\mc}\hat{G}(w)e^{-|w|^2}| (h_m )_I(w)|^2 d\sigma(u,v)\nonumber\\&\rightarrow&0,\;\;\mbox{as}\;m\rightarrow \infty.\end{eqnarray*}
Taking the above calculations into \eqref{I1+I2}, we conclude that $\|W_{f,\varphi} h_m\|\rightarrow 0$ as $m\rightarrow \infty.$ That is to say $W_{f,\varphi}$ is compact on $\mf,$ which completes the proof.
\zb

\section{ Isometric weighted composition operators}

Recall that an operator $T$ on $\mf$ is called isometric if $||Th||=||h||$ for all $h\in \mf$, or equivalently, $\la Th, Tg\ra =\la h, g\ra$ for all $h, g\in \mf$. In this section, we describe all isometric weighted composition operators on the $\mf$ in the slice setting. In what follows, we always denote  $ d\mu(z):= e^{-|z|^2}d\sigma(x,y).$
 \begin{lem} \cite[Lemma 4.1]{Le}\label{lem0} Let $u$ be a measurable function such that $|u(z)|\leq C e^{|z|^2-\epsilon |z|}$ for some constants $\epsilon>0$ and $C>0.$ If for all integers $m,k\geq 0,$ $$\int_{\mathbb{C}} u(z)z^m \overline{z}^k d\mu(z)=0,$$ then $u=0$ a.e. on $\mathbb{C}$. \end{lem}

\begin{prop} \label{prop pla}Let $\eta(p)=p\lambda$ with $|\lambda|\leq 1$ and $\lambda\in \mc$ for some $I\in \ms.$ Denote $\xi\in \mf$ with restriction $\xi_I(z)=F(z)+G(z)J$ on $\mc,$ then  $W_{\xi,\eta}$ is an isometry on $\mf$ if and only if  $|\lambda|=1$,  $F(z)$ and $G(z)$ are constants and \[|\xi_I(z)|^2=|F(z)|^2+|G(z)|^2=1\] which implies $\xi(q)$ is a constant quaternion with modulus one on $\hh$.  \end{prop}
\pf \emph{Sufficiency}. 
We express $h(p)$ as $h(p)=\sum_{k=0}^\infty p^ka_k.$  Equation \eqref{h-norm}  together with \cite[Proposition 3.11]{ACSS} yield
\begin{eqnarray*} &&\|W_{\xi,\eta} h\|^2\\&=&\int_{\mc}|[\xi\star (h\circ \eta)]_I(z)|^2 e^{-|z|^2} d\sigma(x,y)\nonumber\\&=&\int_{\mc}|\xi_I(z)\star h_I (\lambda z)|^2 e^{-|z|^2} d\sigma(x,y)\nonumber\\&=& \int_{\mc}\left|(F +G J)\star \left(\sum_{k=0}^\infty z^k \lambda^k a_k\right) \right|^2  e^{-|z|^2} d\sigma(x,y)\nonumber\\&=&\int_{\mc}\left| F   \left(\sum_{k=0}^\infty z^k \lambda^k a_k\right)+ G \left(\sum_{k=0}^\infty z^k \overline{\lambda^k a_k}\right)J \right|^2  e^{-|z|^2} d\sigma(x,y)\nonumber\\&=&\int_{\mc}\left[ | F |^2 \left|\sum_{k=0}^\infty z^k \lambda^k a_k\right|^2 + |G|^2 \left|\sum_{k=0}^\infty z^k \overline{\lambda^k a_k}\right|^2  \right]e^{-|z|^2} d\sigma(x,y)\nonumber\\&=& |F|^2 \sum_{k=0}^\infty|a_k|^2k!+|G|^2 \sum_{k=0}^\infty|a_k|^2k!\nonumber\\&=& (|F|^2+|G|^2)\|h\|^2=\|h\|^2,\end{eqnarray*} where the fourth line is due to \eqref{starH}.  Hence  the operator $W_{\xi,\eta}$ is an isometry on $\mf$.

\emph{Necessity}.
Suppose $W_{\xi,\eta}$ is an isometry on $\mf$. We have $\lambda\neq 0.$ Indeed, if $\lambda=0,$ then $\eta_I(z)=0,$
 and further $W_{\xi,\eta} p^N=0 $ for all monomials $p^N$  with $N\geq 1,$ which implies  $W_{\xi,\eta}$ is not an isometry on $\mf$.

 In the sequel, we suppose  $\lambda \neq 0$. Picking two slice regular functions  $g_I(z)=z^k$ and $h_I(z)=z^m\in \mf$ with integers $k,m\geq 0,$  it turns out that \begin{eqnarray}&& \la W_{\xi,\eta}h, W_{\xi,\eta} g\ra \nonumber\\&=&\int_{\mc} \overline{(W_{\xi,\eta} g)_I(z)} (W_{\xi,\eta} h)_I(z) e^{-|z|^2} d\sigma(x,y)\nonumber\\&=&   \int_{\mc} \overline{\xi_I(z)\star g_I(\lambda z)} \xi_I(z)\star h_I(\lambda z)  d\mu(z) \nonumber\\&=&  \int_{\mc} \overline{\xi_I(z)\star  (\lambda z)^k} \xi_I(z)\star (\lambda z)^m    d\mu(z).\label{Wgh}\end{eqnarray}
 Since  $\xi_I(z)\star  (\lambda z)^k=F(z)\lambda^k z^k+G(z)\overline{\lambda}^k z^k J$ and $J F(z)=\overline{F(z)}J,$   the display \eqref{Wgh} turns into
 \begin{eqnarray}  &&   \int_{\mc} \overline{\xi_I(z)\star  (\lambda z)^k} \left(\xi_I(z)\star (\lambda z)^m \right) d\mu(z)\nonumber\\&=&  \int_{\mc} \left[(\overline{F(z)}\overline{\lambda}^k-J\overline{G(z)} \lambda^k)\overline{z}^k\right]\cdot\left[ F(z)\lambda^m z^m +G(z)\overline{\lambda}^m z^mJ\right]    d\mu(z)\nonumber\\&=&  \int_{\mc} (|F(z)|^2\overline{\lambda}^k\lambda^m \overline{z}^k z^m-J|G(z)|^2\lambda^k \overline{\lambda}^m\overline{z}^k z^m J) d\mu(z)\nonumber\\&&+  \int_{\mc} (\overline{F(z)} G(z)\overline{\lambda}^k\overline{\lambda}^m \overline{z}^k z^mJ-J\overline{G(z)}F(z) \lambda^k \lambda^m\overline{z}^k z^m ) d\mu(z)\nonumber\\ &=&  \int_{\mc} (|F(z)|^2\overline{\lambda}^k\lambda^m \overline{z}^k z^m-J^2|G(z)|^2\overline{\lambda}^k \lambda^mz^k \overline{z}^m )d\mu(z)\nonumber\\&&+ \int_{\mc} \left( \overline{F(z)} G(z)\overline{\lambda}^k\overline{\lambda}^m \overline{z}^k z^mJ - G(z)\overline{F(z)} \overline{\lambda}^k \overline{\lambda}^mz^k \overline{z}^m J\right)  d\mu(z) \nonumber\\&=& \int_{\mc} |F(z)|^2\overline{\lambda}^k\lambda^m \overline{z}^k z^m d\mu(z)  + \int_{\mc}|G(\overline{z})|^2 \overline{\lambda}^k\lambda^m \overline{z}^k z^m  d\mu(z) \nonumber\\&&+ \int_{\mc} \overline{F(z)} G(z)\overline{\lambda}^k\overline{\lambda}^m \overline{z}^k z^mJd\mu(z) - \int_{\mc} G(\overline{z})\overline{F(\overline{z})} \overline{\lambda}^k \overline{\lambda}^m\overline{z}^k z^m J  d\mu(z)\nonumber\\&=& \overline{\lambda}^k\lambda^m \int_{\mc}\left( |F(z)|^2 +|G(\overline{z})|^2 \right) \overline{z}^k z^m d\mu(z) \nonumber\\&&+ \overline{\lambda}^k \overline{\lambda}^m\int_{\mc} \left(\overline{F(z)} G(z)  -  G(\overline{z})\overline{F(\overline{z})}\right)\overline{z}^k z^m J  d\mu(z). \label{Wgh1} \end{eqnarray} On the other side,  we use  the change of variable $w=\lambda z$ on  the second line of following equations and  then obtain that
 \begin{eqnarray}  \la h ,g\ra&=&\int_{\mc} \overline{g_I(w)}h_I(w) e^{-|w|^2} d\sigma(u,v)\nonumber\\&=&|\lambda|^2 \int_{\mc} \overline{g_I(\lambda z)}h_I(\lambda z) e^{-|\lambda z|^2} d\sigma(x,y)\nonumber\\&=&
 |\lambda|^2 \int_{\mc} e^{-|\lambda|^2 |z|^2}\overline{\lambda}^k \overline{z}^k \lambda^m z^m d\sigma(x,y)\nonumber\\&=&  |\lambda|^2\overline{\lambda}^k  \lambda^m \int_{\mc} e^{(1-|\lambda|^2) |z|^2} \overline{z}^k z^m e^{-|z|^2}d\sigma(x,y)\nonumber\\&=&  |\lambda|^2\overline{\lambda}^k  \lambda^m \int_{\mc} e^{(1-|\lambda|^2) |z|^2} \overline{z}^k z^m d\mu(z).\label{hg}\end{eqnarray}

As $W_{\xi,\eta}$  is an  isometry on $\mf$ which ensures  $\la W_{\xi,\eta} h, W_{\xi,\eta} g\ra= \la h,g\ra$, hence employing \eqref{Wgh}, \eqref{Wgh1} and  \eqref{hg}, we deduce that
\begin{eqnarray*} && |\lambda|^2\overline{\lambda}^k  \lambda^m \int_{\mc} e^{(1-|\lambda|^2) |z|^2} \overline{z}^k z^m d\mu(z)\nonumber\\&=& \overline{\lambda}^k\lambda^m \int_{\mc}\left( |F(z)|^2 +|G(\overline{z})|^2 \right)   \overline{z}^k z^m  d\mu(z) \nonumber\\&&+ \overline{\lambda}^k \overline{\lambda}^m \int_{\mc} \left(\overline{F(z)} G(z)  -  G(\overline{z})\overline{F(\overline{z})}\right) \overline{z}^k z^m J    d\mu(z).\end{eqnarray*}
The above equations ensure that
\begin{eqnarray}&&\int_{\mc}   \left( |F(z)|^2 +|G(\overline{z})|^2-|\lambda|^2 e^{(1-|\lambda|^2)|z|^2} \right)  \overline{z}^k z^m d\mu(z)=0, \label{equal1}\\&&\int_{\mc} \left(\overline{F(z)} G(z)  -  G(\overline{z})\overline{F(\overline{z})}\right)  \overline{z}^k z^m J   d\mu(z)=0.\label{equal2}
 \end{eqnarray} By Theorem \ref{thm bou}, it holds that \begin{eqnarray*}|F(z)|^2+|G(\overline{z})|^2&\leq& C e^{|z|^2-|\lambda z|^2}\\&=&Ce^{(1-|\lambda|^2)|z|^2}\end{eqnarray*} for some constant $C>0.$
Now formula \eqref{equal1}  and Lemma \ref{lem0} imply that
\begin{eqnarray}|F(z)|^2+|G(\overline{z})|^2&=&|\lambda|^2 e^{(1-|\lambda|^2)|z|^2}\;\;\;\mbox{for all}\;z\in \mc,\label{con1}\\ \overline{F(z)} G(z)  -  G(\overline{z})\overline{F(\overline{z})} & =&0  \;\;\;\mbox{for all}\;z\in \mc.\label{con2}\end{eqnarray}
The display \eqref{con2} implies
\begin{eqnarray} \frac{G(\overline{z})}{\overline{F(z)}}=\frac{G(z)}{\overline{F(\overline{z})}}.\label{con22} \end{eqnarray} which entails both sides is  a constant denoting by $\delta$ on $\mc.$ That is to say $G(\overline{z})=\delta \overline{F(z)}$. Hence the display \eqref{con1} turns into
\begin{eqnarray} |F(z)|^2+|\delta|^2|\overline{F(z)}|^2=(1+|\delta|^2)|F(z)|^2=|\lambda|^2 e^{(1-|\lambda|^2)|z|^2}\;\;\;\label{con11} \end{eqnarray} $\mbox{for all}\;z\in \mc.$ Furthermore, it holds that
$$ (1+|\delta|^2) F(z)\overline{F(w)} =|\lambda|^2 e^{(1-|\lambda|^2)z\overline{w}}\;\;\;\mbox{for all}\;z\in \mc.$$ Setting $w=0,$ we deduce that \[(1+|\delta|^2) F(z)\overline{F(0)}=|\lambda|^2,\] which implies $F(z)$  is a constant function  and then $|\lambda|=1$ in \eqref{con11}. Consequently, $G(z)=\delta \overline{F(\overline{z})}$ is also a constant function. Moreover, observing \eqref{con1}, it turns out \[|F(z)|^2+|G(z)|^2=|f_I(z)|^2=1\] which completes the proof.  \zb
For the case $|\lambda|=1$, we construct a unitary operator $W_{k_{\overline{\lambda} b},\varphi}$  on $\mf$ by $\varphi(p)=p\lambda-b$.
\begin{prop}\label{prop uni}Let $\varphi(p)=p\lambda-b$ with $|\lambda|=1$ and $\lambda,\;b\in \mc$ such that $\varphi(\mc)\subset \mc$ for some $I\in\ms.$ Then the weighted composition operator $W_{k_{\overline{\lambda} b},\varphi}$ is a unitary operator in $\mf$ and it also holds that \[W_{k_{\overline{\lambda} b},\varphi}^{-1}=W_{k_{\overline{\lambda} b},\varphi}^*=W_{k_{-b},\varphi^{-1}}.\] \end{prop}
 \pf Let $h\in \mf,$ and $w=z\lambda-b\in \mc, $ denote by $w=u+Iv$ and $z=x+Iy$. We have
  \begin{eqnarray*}\|h\|^2&=&\int_{\mc}|h_I(w)|^2 e^{-|w|^2} d\sigma(u,v)\nonumber\\&=&\int_{\mc}|h_I(z\lambda -b)|^2 e^{-|z\lambda -b|^2} d\sigma(x,y)\nonumber\\&=&\int_{\mc}|h_I(z\lambda -b)|^2 |e^{ \lambda \overline{b} z}|^2 e^{-|b|^2-|z|^2} d\sigma(x,y)\nonumber\\&=&\int_{\mc}\left|e^{\lambda \overline{b}z-|b|^2/2}\right|^2 |h_I(z\lambda -b)|^2 |e^{ -|z|^2} d\sigma(x,y)\nonumber\\&=& \int_{\mc}|k_{\overline{\lambda} b}(z)|^2 |h_I(z\lambda -b)|^2 |e^{ -|z|^2} d\sigma(x,y) \nonumber\\&=&\int_{\mc}|k_{\overline{\lambda} b}(z)\star h_I(z\lambda -b)|^2 |e^{ -|z|^2} d\sigma(x,y) \nonumber\\&=&\|W_{k_{\overline{\lambda}b},\varphi} h\|^2,
  \end{eqnarray*}
  which implies $W_{k_{\overline{\lambda}b},\varphi}$ is an isometry and \[W_{k_{\overline{\lambda}b},\varphi}^* W_{k_{\overline{\lambda}b},\varphi}=I,\] the identity map. On the other hand,  for $h\in \mf$ and $z\in \mc,$ it yields that
  \begin{eqnarray} &&[W_{k_{\overline{\lambda} b},\varphi} W_{k_{-b},\varphi^{-1}} h_I](z)\nonumber\\&=&W_{k_{\overline{\lambda} b},\varphi}(W_{k_{-b},\varphi^{-1}} h_I(z)) \nonumber\\&=& k_{\overline{\lambda} b}(z) \star (W_{k_{-b,\varphi^{-1}}} h_I(\varphi_I(z)))\nonumber\\&=& k_{\overline{\lambda} b}(z) \star k_{-b}(\varphi_I(z))\star h_I(z)\nonumber\\&=&h_{I}(z),\nonumber \end{eqnarray}
  which entails  $W_{k_{\overline{\lambda} b},\varphi} W_{k_{-b},\varphi^{-1}}=I.$ Hence $W_{k_{\overline{\lambda} b},\varphi} $ is a unitary operator whose inverse is
 $W_{k_{-b},\varphi^{-1}}.$ This completes the proof.
  \zb
\begin{rem} \label{rem Wu}In the case $\varphi_1(p)=p-u$ with $u\in \mc,$ we denote
\begin{eqnarray}W_{u}=W_{k_u,\varphi_1}\label{wu}\end{eqnarray} a special weighted composition operator with weight $k_u$ and composition symbol $\varphi_1(p)=p-u.$ They satisfy the commutation relation: \begin{eqnarray*}W_uW_v=e^{i Im (u\overline{v})}W_{u+v},\label{commutation}\end{eqnarray*}which implies $W_{u}^{-1}=W_{-u}.$\end{rem}
We are now in a position to describe all isometric weighted composition operators on $\mf$.
\begin{thm}\label{iso} Let $f$ and $\varphi$ be two slice regular functions on $\mathbb{H},$ such that $f_I(z)=F(z)+G(z)J$ is not identically zero  and $\varphi(\mathbb{C}_I) \subset \mathbb{C}_I$ for some $I\in \mathbb{S}$.  Then $W_{f,\varphi}$ is isometric on $\mf$ if and only if  $$(W_{f,\varphi} h)_I(z)= (\alpha +\beta  J)\star( W_{k^2_{-\overline{\lambda} b},\varphi}h)_I(z),$$  with $\alpha, \beta \in \mc$ and $|\alpha|^2+|\beta|^2=1, $ for any $h\in \mf.$
\end{thm}
\pf
\emph{Necessity}.  Suppose the operator $W_{f,\varphi}$ is isometric on $\mf$, then it is bounded, hence $\varphi(p)=b+p\lambda$ for some $|\lambda|\leq 1$, $\lambda,\;b\in \mc$ by  Theorem \ref{thm bou}.

We denote $W_{\hat{f},\hat{\varphi}}=W_{f,\varphi} W_b,$ where $W_b$ defined in \eqref{wu} with $b\in \mc$ and $\varphi_1(p)=p-b.$  After that,  for any $h_I(z)=H(z)+K(z)J\in \mf,$ we deduce that  \begin{eqnarray*} &&(W_{\hat{f},\hat{\varphi}} h)_I(z)\\&=&(W_{f,\varphi} W_bh)_I(z)\\&=&[F(z)k_b(b+z\lambda)+G(z)\overline{k_b(b+\overline{z}\lambda)} J]\star [H(z\lambda)+K(z\lambda)J].\end{eqnarray*} That is to say  $\hat{\varphi}(p)=p\lambda$ and \[\hat{f}_I(z)=F(z)k_b(b+z\lambda)+G(z)\overline{k_b(b+\overline{z}\lambda)}.\] Moreover, if the operator $W_{\hat{f},\hat{\varphi}}$ is isometric, Proposition \ref{prop pla} entails that $|\lambda|=1$ and \[|\hat{f}_I|^2=|F(z)k_b(b+z\lambda)|^2+|G(z)\overline{k_b(b+\overline{z}\lambda)}|^2=1,\] with $F(z)k_b(b+z\lambda)$ and $G(z)\overline{k_b(b+\overline{z}\lambda)}$ are constants denoted by $\alpha\in \mc$ and $\beta\in\mc$, respectively. Therefore,   $$F(z)=\alpha k_{-\overline{\lambda} b}(z)\;\;\mbox{ and}\;\;G(z)=\beta k_{-\lambda \overline{b}}(z).$$ This means \[\hat{f}_I(z)=(\alpha+\beta J)\star k_{-\overline{\lambda} b}(z).\]
Hence, by Proposition \ref{prop uni} and Remark \ref{rem Wu}, we obtain \[W_{f,\varphi}=W_{\hat{f},\hat{\varphi}} W_{-b},\] which can be also rewritten as
\begin{eqnarray*}&&(W_{f,\varphi}h)_I(z)\\&=&(W_{\hat{f},\hat{\varphi}} W_{-b} h)_I(z)\nonumber\\&=&\hat{f}_I(z)\star[k_{-b}(\lambda z)(H(z\lambda+b)+K(z\lambda +b)J)]\\&=&(\alpha+\beta J)\star k_{-\overline{\lambda} b}(z)\star[k_{-b}(\lambda z)(H(z\lambda+b)+K(z\lambda +b)J)]\\&=& (\alpha +\beta  J) \star k^2_{-\overline{\lambda} b}( z) \star (H(z\lambda+b)+K(z\lambda +b)J)\\&=& (\alpha +\beta  J)\star( W_{k^2_{-\overline{\lambda} b},\varphi}h)_I(z),\end{eqnarray*}  with $\alpha, \beta \in \mc$ and $|\alpha|^2+|\beta|^2=1.$

\emph{Sufficiency}. Suppose $$(W_{f,\varphi} h)_I(z)= (\alpha +\beta  J)\star( W_{k_{-\overline{\lambda} b},\varphi}h)_I(z),$$  with $\alpha, \beta \in \mc$ and $|\alpha|^2+|\beta|^2=1, $ for any $h\in \mf.$ It holds that \[W_{f,\varphi}=W_{\hat{f},\hat{\varphi}} W_{-b}.\] For any element $h\in \mf,$ there exists $g\in \mf$ such that \[g=W_{-b} h.\]  Based on the fact $W_{-b}$ is a unitary operator on $\mf$, it follows that $\|h\|=\|g\|.$ Since $W_{\hat{f},\hat{\varphi}}$ is isometric on $\mf,$ it yields that
$$\|W_{f,\varphi}h\|=\|W_{\hat{f},\hat{\varphi}} W_{-b} h\|=\|W_{\hat{f},\hat{\varphi}} g\|=\|g\|=\|h\|.$$ That is to say
$W_{f,\varphi}$ is isometric on $\mf$. This ends the proof. \zb

\section{conjugate $\mathcal{C}_{a,b,c}-$commuting weighted composition operators on $\mf$}
Based on the previous work in Section 3, we  define a general weighted composition right-conjugation $\mathcal{C}_{a,b,c,d}$ on $\mf$. Then we give equivalent conditions ensuring the weighted composition operators which commute  with respect to our conjugation $\mathcal{C}_{a,b,c}$ on $\mf.$ First, we introduce the definition of the  right-conjugation (left-conjugation) and conjugation  in general Hilbert space $\mathcal{H}$ on $\hh.$
\begin{defn}\label{defn anti} A mapping $T$ acting on a quaternion Hilbert space $\mathcal{H}$ is said to be right-anti-complex-linear (also right-conjugate-linear) associated to the slice $\mc$, if $T(pa+qb)=T(p)\overline{a}+T(q)\overline{b},$ for any $ p, q\in \mathcal{H} \;\;\mbox{and any}\; a,b\in \mc. $ Similarly, $T$ on $\mathcal{H}$   is left-anti-linear (also left-conjugate-linear) in $\hh$, if $T(ap+bq)=\overline{a}T(p)+\overline{b}T(q),$ for any $ p, q\in \mathcal{H} \;\;\mbox{and any}\; a,b\in \mc. $  Furthermore, $T$ on $\mathcal{H}$ is anti-linear in $\hh$, if $T$ is both left-anti-linear and right-anti-linear. \end{defn}

In order to construct complex symmetric operators in the quaternion Fock space, we call a map is linear which means that it is complex linear associated with $\mc$, nor quaternion linear.
\begin{defn} \label{conjugation}A right anti-(complex)-linear (left-anti-linear) mapping $\mathcal{C}:\;\mathcal{H}\rightarrow \mathcal{H}$ is called a right-conjugation (left-conjugation), if it is \\
$(i)$ involutive:\; $\mathcal{C}^2=I$, the identity operator;\\
$(ii)$ isometric:\; $\|\mathcal{C} f\|=\|f\|$, for all $f\in \mathcal{H}.$ \\
 Furthermore, an anti-linear mapping $\mathcal{C}$ which is involutive and   isometric is called a
 \emph{conjugation}. \end{defn}

Firstly, we  define a mapping \begin{eqnarray}\mathcal{C}f_I(z)=\overline{F(\overline{z})}+\overline{G(\overline{z})}J \label{C}\end{eqnarray}  for $f_I(z)=F(z)+G(z)J$ in $\mf.$  It's easy to check that \begin{eqnarray*}\mathcal{C}(f_I(z)a)&=&\overline{F(\overline{z}) a}+\overline{G(\overline{z})\overline{a}}J\\&=&(\overline{F(\overline{z})}+ \overline{G(\overline{z})}J)\overline{a}\\&=&\mathcal{C}f_I(z)\overline{a}\end{eqnarray*}
for any $a\in \mc.$
 It  turns out that
 \begin{eqnarray}\mathcal{C}( f_I(z)a+g_I(z)b)=  \mathcal{C} f_I(z)\overline{a}+  \mathcal{C} g_I (z) \overline{b}\label{antic}\end{eqnarray} and also holds that $$\mathcal{C}( af_I(z)+bg_I(z))= \overline{a} \mathcal{C} f_I(z)+  \overline{b} \mathcal{C} g_I (z).$$ The above two inequalities entail $\mathcal{C}$ is anti-linear on  $\mf.$    Meanwhile, $  \mathcal{C}^2 f_I=f_I.$   That is to say $\mathcal{C}$ is  anti-linear  and involutive on $\mf$. Furthermore, \begin{eqnarray*}\|\mathcal{C}f\|^2&=&\int_{\mc} |(\mathcal{C}f)_I(z)|^2 e^{-|z|^2}d\sigma(x,y)\\&=&\int_{\mc}( |\overline{F(\overline{z}) }|^2 +|\overline{G(\overline{z}) }|^2 )e^{-|z|^2}d\sigma(x,y)\nonumber\\&=&\int_{\mc}( |F(\overline{z})|^2 +|G(\overline{z}) |^2 )e^{-|\overline{z}|^2}d\sigma(x,y)\\&=&\|f\|^2,\end{eqnarray*} which implies
 $\|\mathcal{C} f\|=\|f\|$   on $\mf.$ The above facts ensure  the mapping $\mathcal{C}$ is a\emph{ conjugation} on $\mf.$
  It's also true that \begin{eqnarray}&&\mathcal{C}(f_I\star g_I) = \mathcal{C}f_I\star \mathcal{C} g_I\nonumber\\&=&[\overline{F(\overline{z})}\;\overline{H(\overline{z})}-
 \overline{G(\overline{z})}K(z)]+[\overline{F(\overline{z})}\;\overline{K(\overline{z})}+
 \overline{G(\overline{z})}H(z)]J.\label{CC} \end{eqnarray}
 Now we introduce some definitions for   (right-)conjugate  $\mathcal{C}-$commuting  operator and  (right-)complex $\mathcal{C}-$symmetric  operator.
  \begin{defn} \label{complex-c} A bounded linear operator $T$ on $\mathcal{H}$ is called a right-conjugate (or  left-conjugate) $\mathcal{C}-$commuting  operator, if there exists a right-conjugation (or left-conjugation) $\mathcal{C}$ on $\mathcal{H}$ such that $\mathcal{C}T=T \mathcal{C}.$ A operator $T$ is called a conjugate $\mathcal{C}-$commuting  operator if   the mapping $\mathcal{C}$ is a conjugation.\end{defn}
  \begin{defn}\label{complex-s} A bounded linear operator $T$ on $\mathcal{H}$ is called  a right-complex $\mathcal{C}-$symmetric (or left-complex $\mathcal{C}-$symmetric) operator, if there exists a right-conjugation (or left-conjugation) $\mathcal{C}$ on $\mathcal{H}$ such that $T\mathcal{C}=\mathcal{C}T^*.$  Furthermore,  $T$ is  called  complex $\mathcal{C}-$symmetric operator if the mapping  $\mathcal{C}$ is a conjugation.
  \end{defn}
Considering the concrete conjugation $\mathcal{C}$ defined in \eqref{C} on the space $\mf$, we offer the equivalent descriptions for weighted composition operator  $W_{f,\varphi}$  is  a conjugate $\mathcal{C}-$commuting  operator on $\mf.$
 \begin{thm}\label{thm C-}Let $f$ and $\varphi$ be two slice regular functions on $\hh$, such that $f$ is not identically zero and $\varphi(\mc)\subset\mc$ for some $I\in\ms.$ Suppose the operator $W_{f,\varphi}$ is bounded on $\mf,$ that is,  $\varphi(p)=pA+B$ with $|A|\leq 1$ and $A \in \mc.$ Then $W_{f,\varphi}$  is  a conjugate  $\mathcal{C}-$commuting operator if and only if
 \begin{eqnarray}f_I=\mathcal{C} f_I,\;\overline{A}=A\;\mbox{and}\;\overline{B}=B.\label{c-comm} \end{eqnarray}\end{thm}

 \pf \emph{Necessity.}  Suppose  $W_{f,\varphi}$  is  a conjugate $\mathcal{C}-$commuting operator, that is \begin{eqnarray}W_{f,\varphi} \mathcal{C} h=\mathcal{C} W_{f,\varphi}h\label{comm} \end{eqnarray} for any $h\in \mf.$ Particularly, taking $h=1$   in \eqref{comm}, it yields that
 \begin{eqnarray*} W_{f_I,\varphi_I} \mathcal{C} 1=W_{f_I,\varphi_I} 1=f_I \;\;\mbox{and}\;\;\mathcal{C} W_{f_I,\varphi_I} 1=\mathcal{C}f_I.\end{eqnarray*} The above formulas entail \begin{eqnarray}f_I=\mathcal{C} f_I.\label{Cfi}\end{eqnarray}
 On the other side, letting $h_I(z)=z$ in \eqref{comm} and employing \eqref{Cfi}, we deduce
 \begin{eqnarray*} W_{f_I,\varphi_I} \mathcal{C} z&=&W_{f_I,\varphi_I} z=f_I\star (Az+B),\\ \mathcal{C} W_{f_I,\varphi_I} z&=&\mathcal{C}[f_I\star (Az+B)]\\&=&\mathcal{C}(f_I)\star(\overline{A}z+\overline{B})
\\& =&f_I\star(\overline{A}z+\overline{B}).\end{eqnarray*}
  That is to say \[f_I\star (Az+B)=f_I\star(\overline{A}z+\overline{B}).\]Assuming $f_I(z)=F(z)+G(z)J$, it turns out
  \begin{eqnarray} &&F(z)(Az+B)+G(z)(\overline{A}z+\overline{B}) J\nonumber\\&=& F(z)(\overline{A}z+\overline{B})+G(z)(Az+B) J.~~\label{ab}\end{eqnarray}
 Since $f$ is not identically zero, therefore  at least one of $F$ and $G$ is not identically zero. The formula \eqref{ab} entails that $Az+B=\overline{A}z+\overline{B},$ that is $\overline{A}=A$ and $\overline{B}=B.$

\emph{ Sufficiency.} Suppose the conditions \eqref{c-comm} holds. For any $h\in \mf$ with $h_I(z)=H(z)+K(z)J$ on $\mc,$ it follows that
\begin{eqnarray*} &&W_{f,\varphi} \mathcal{C} h_I(z)\nonumber\\&=&f_I(z)\star [\overline{H(\overline{Az+B})}+\overline{G(\overline{Az+B})}J] \nonumber\\&=&\mathcal{C}f_I(z)\star [\overline{H( A\overline{z}+B)}+\overline{G(A\overline{z}+B)}J]\nonumber
\\&=&\mathcal{C}f_I(z)\star \mathcal{C}[ H(Az+B)+ G(A z+B)J] \nonumber\\&=&
\mathcal{C}[f_I\star (H( Az+B)+ G(A z+B)J) ]
\nonumber\\&=&\mathcal{C} W_{f,\varphi} h_I(z),\end{eqnarray*} which reveals that $W_{f,\varphi}$  is a conjugate $\mathcal{C}-$commuting operator on $\mf$.  \zb

   In the following, we define an isometry  $A_b:\;\mf \rightarrow \mf$ by \begin{eqnarray*}A_b f_I(z)&=&k_{-b}(z)\star (\overline{F(\overline{z+b})})+\overline{G(\overline{z+b})})J \nonumber\\&=&k_{-b}(z)\star ( \mathcal{C} f_I)(z+b)
 ,\end{eqnarray*}  where $k_b= K_b/ \|K_b\| $ the unit vector in $\mf$ with $b\in \mc.$ The following lemma reveals   the operator $A_b$  is isometric on $\mf.$

 \begin{lem} \label{lem Ab}For any $b\in \mc,$ $A_b$ is isometric on $\mf.$ \end{lem}
\pf Using the equation \begin{eqnarray}|k_{-b}(u-b)|^2 e^{-|u-b|^2}=e^{-|u|^2}\label{K}\end{eqnarray} for $b, u\in \mc,$ we deduce that
\begin{eqnarray*} \|A_bf\|^2&=&\int_{\mc}|(A_b f)_I(z)|^2  e^{-|z|^2} d\sigma(x,y)\nonumber\\&=& \int_{\mc}|k_{-b}(z)\star (\overline{F(\overline{z+b})})+\overline{G(\overline{z+b})})J)|^2 e^{-|z|^2} d\sigma(x,y)\nonumber\\&=&\int_{\mc} |k_{-b}(z)\cdot \overline{F(\overline{z+b})}+k_{-b}(z)\cdot \overline{G(\overline{z+b})})J)|^2 e^{-|z|^2} d\sigma(x,y)\nonumber\\&=&\int_{\mc} |k_{-b}(z)|^2 \left[|F(\overline{z+b})|^2+|G(\overline{z+b})|^2\right] e^{-|z|^2} d\sigma(x,y)\nonumber\\&=&\int_{\mc} |k_{-b}(w-b)|^2 \left[|F(\overline{w})|^2+|G(\overline{w})|^2\right] e^{-|w-b|^2} d\sigma(u,v)\nonumber\\&=& \int_{\mc} |k_{-b}(w-b)|^2 \left|\mathcal{C} f_I(w)\right|^2 e^{-|w-b|^2} d\sigma(u,v)\nonumber\\&=&
\int_{\mc}\left|\mathcal{C} f_I(w)\right|^2 e^{-|w|^2}d\sigma(u,v)\nonumber\\&=&\|\mathcal{C} f\|^2=\| f\|^2.\end{eqnarray*} \zb

At present,  given two  functions $\xi,\;\eta \in \mathcal{R}(\hh)$ with $\xi$ not identically zero and $\eta_I(\mc)\subset \mc$ for some $I\in \ms,$ we define a  weighted composition operator, \begin{eqnarray}(A_{\xi,\eta}f_I)(z)&=&\xi_I(z) \star [\overline{F(\overline{\eta_I(z)})}+\overline{G(\overline{\eta_I(z)})}J]\nonumber\\&=&\xi_I(z) \star (\mathcal{C}f_I)(\eta_I(z)),\label{anti-liear}
 \end{eqnarray} for $f_I(z)=F(z)+G(z)J\in \mf$ .
 Indeed, picking $g_I(z)=H(z)+K(z)J\in \mf$ and using \eqref{antic}   we deduce
   \begin{eqnarray} && (A_{\xi,\eta}( f_Ia+ g_I b))(z)\nonumber\\&=&\xi_I(z) \star [\mathcal{C}( f_I a+ g_I b)](\eta_I(z))\nonumber\\&=&\xi_I(z) \star [(\mathcal{C}f_I)(\eta_I(z))\overline{a}+(\mathcal{C}g_I)(\eta_I(z))\overline{b}]\nonumber
   \\&=& \xi_I(z) \star (\mathcal{C}f_I)(\eta_I(z)) \overline{a} +\xi_I(z) \star (\mathcal{C}g_I)(\eta_I(z))\overline{b}\nonumber\\&=&(A_{\xi,\eta}f_I)(z)\overline{a}
 +(A_{\xi,\eta}g_I)(z)\overline{b}\end{eqnarray} which together with
   Definition \ref{defn anti} validate $A_{\xi,\eta}$ is right-anti-linear  weighted composition operator.

 In the sequel, we  continue to find the equivalent characterizations for its boundedness, isometry and involution on $\mf.$  It's clear that $A_{\xi,\eta}$ is $\mf$-invariant if and only if it is bounded on $\mf.$ Compared with the arguments in Theorem \ref{thm bou}, we obtain a similar result for the boundedness of the operator  $A_{\xi,\eta}$   on $\mf.$
 \begin{prop}\label{prop-anti} The  right-anti-linear  weighted composition operator $ A_{\xi,\eta}:\;\mf \rightarrow \mf$ is bounded if and only if

 (1)$\xi \in \mf;$
(2)
\begin{eqnarray*}\hat{M}(\xi,\eta)&=&\sup\limits_{w\in \mc}\left(|F(w)|^2 e^{|\eta_I(w)|^2-|w|^2}+ |G(w)|^2e^{|\eta_I(\overline{w})|^2-|w|^2}\right)\\&=&\sup\limits_{w\in \mc} \left( \xi_{I}(w)\star e^{\frac{|\eta_I(w)|^2-|w|^2}{2}}    \right)
  <\infty,\end{eqnarray*}
 where $\xi_I(w)=F(w)+G(w)J$ with $I\perp J.$  Besides, in this case, $\eta(p)=pa+b$ satisfying  $|a|\leq 1,$ $a\in \mc.$\end{prop}

 Indeed, if the operator $ A_{\xi,\eta}:\;\mf \rightarrow \mf$ is bounded with \emph{non-constant function} $\eta$ such that $\eta(\mc)\subset\mc$, then it yields that
 \begin{eqnarray*}\|\xi\|^2&=&\int_{\mc}(|F(z)|^2+|G(z)|^2)e^{-|z|^2}d\sigma(x,y) \nonumber\\&=&\int_{\mc}|F(z)|^2e^{-|z|^2}d\sigma(x,y) +\int_{\mc}|G(\overline{z})|^2e^{-|z|^2}d\sigma(x,y) \nonumber\\&=& \int_{\mc}|F(z)|^2e^{|\eta_I(z)|^2-|z|^2}e^{-|\eta_I(z)|^2}d\sigma(x,y) \nonumber\\&&+\int_{\mc}|G(\overline{z})|^2e^{|\eta_I(z)|^2-|z|^2}e^{-|\eta_I(z)|^2}d\sigma(x,y)
 \nonumber\\&\leq& \hat{M}(\xi,\eta) \int_{\mc}e^{-|az+b|^2}d\sigma(x,y)\nonumber\\ &=&
 \frac{1}{|a|^2}\hat{M}(\xi,\eta) \int_{\mc}e^{-|z|^2}d\sigma(x,y)<+\infty. \end{eqnarray*} That is to say the statements $(2)\Rightarrow (1)$ in Proposition \ref{prop-anti} if $\eta$ is non-constant function with $0<|a|\leq 1$.

 Based on the above calculations, we continue  exploring the conditions ensuring $A_{\xi,\eta}$ is an isometry  on $\mf.$ In what follows, we always assume that \emph{$\xi$ and $\eta$ are two slice regular functions, where nonconstant function $\eta$  satisfies $\eta(\mc)\subset \mc$ for some $I\in \ms,$ and   $\hat{M}(\xi,\eta)<\infty $ in Proposition \ref{prop-anti}.} The proposition below indicates some special forms for the functions $\xi$ and $\eta$ when $ A_{\xi,\eta}:\;\mf\rightarrow \mf$ is an isometry.
\begin{prop} \label{prop iso}  The  right-anti-linear  weighted composition operator $A_{\xi,\eta} $ is isometric on $\mf$ if and only if the functions $\xi$ and $\eta$ are of the following forms,
\begin{eqnarray}  \xi_I(z)&=&(c+dJ)\star e^{-a\overline{b}z}=ce^{-a\overline{b}z}+de^{-\overline{a}bz} J, \nonumber\\ \eta_I(z)&=&za+b, \label{xi1}\end{eqnarray} where $a, b, c, d \in \mc$ satisfying \begin{eqnarray} |a|=1,\;\;(|c|^2+|d|^2) e^{|b|^2}=1.\label{abcd}\end{eqnarray}\end{prop}
\pf \emph{Necessity.} Suppose the operator $ A_{\xi,\eta}:\;\mf\rightarrow \mf$ is isometric. By Proposition \ref{prop-anti}, it yields that $\eta(p)=p a+b$ for some constants $a$ and $b$ with $|a|\leq 1,$ $a, b\in \mc.$  For every $f\in \mf,$ we deduce that
\begin{eqnarray*}A_{-\overline{b}}f_I(z)&=&k_{\overline{b}}(z)\star [\overline{F(\overline{z}-b)}+\overline{G(\overline{z}-\overline{b})}J]\\&
=&k_{\overline{b}}(z)\overline{F(\overline{z}-b)}
+k_{\overline{b}}(z)\overline{G(\overline{z}-b)}J.\end{eqnarray*} Furthermore, it entails that
\begin{eqnarray*} &&A_{\xi,\eta} A_{-\overline{b}} f_I(z)\nonumber\\&=& \xi_I(z)\star [\overline{k_{\overline{b}}(\overline{\eta_I(z)})}\overline{\overline{F(\eta_I(z)-b)}}
+\overline{k_{\overline{b}}(\overline{\eta_I(z)})}\overline{\overline{G(\eta_I(z)-b)}}J] \nonumber\\&=&\xi_I(z)\star[k_{\overline{\eta_I(z)}}(\overline{b})F(\eta_I(z)-b)+
k_{\overline{\eta_I(z)}}(\overline{b})G(\eta_I(z)-b)J]\nonumber\\&
=&\xi_I(z)\star[k_{ b}(\eta_I(z))F(az)+k_{ b}(\eta_I(z))
 G(az)J]\nonumber\\&=&\xi_I(z)\star k_{ b}(\eta_I(z))\star f_I(az). \end{eqnarray*}
Setting $u(p)=\xi(p)\star k_b(\eta (p))$ and $v(p)=pa,$ we can define a linear weighted composition operator on $W_{u,v}$ on $\mf$ defined by $$W_{u,v} h(p)=u(p)\star [h(v(p))]$$ and then we see that
\[A_{\xi,\eta} A_{-\overline{b}}  =W_{u,v}.\]  On  the other hand, for every $f\in \mf,$ Lemma \ref{lem Ab} and the isometric operator  $ A_{\xi,\eta}:\;\mf\rightarrow \mf$  entail that
\begin{eqnarray}\|A_{\xi,\eta}A_{-\overline{b}} f\|=\|A_{-\overline{b}} f\|=\|f\|.\label{xietab}\end{eqnarray} That is to say  $W_{u,v}=A_{\xi,\eta} A_{-\overline{b}}$ is isometric. Denote $\xi_I(z)=X(z)+Y(z)J$ for $z\in \mc$, and immediately \begin{eqnarray*}u_I(z)&=&\xi_I(z)\star k_b(\eta_I(z))\\&=&(X(z)+Y(z)J)\star k_b(\eta_I(z))\nonumber\\&=&X(z)k_b(\eta_I(z)) + Y(z)\overline{k_b(\eta_I(\overline{z}))}J\nonumber\\&=& X(z)k_b(\eta_I(z)) + Y(z)\overline{k_b(\eta_I(\overline{z}))}J\end{eqnarray*} Employing Proposition \ref{prop pla}, it entails that $W_{u,v}$ is an isometry on $\mf$, then  $|a|=1$ and $u_I(z)$ satisfying  $\varsigma:=X(z)k_b(\eta_I(z))$ and $\chi:=Y(z)\overline{k_{b}(\eta_I(\overline{z}))}$ are constants and $|\varsigma|^2+|\chi|^2=1.$ Hence
\begin{eqnarray*} &&X(z)= \varsigma e^{-|b|^2/2} e^{-a\overline{b}z}\;\;\mbox{and}\;\; Y(z)=\chi e^{-|b|^2/2} e^{-\overline{a}bz}.\end{eqnarray*} Setting  $c:=\varsigma e^{-|b|^2/2} $ and $d:=\chi e^{-|b|^2/2} $, and then $X(z)= ce^{-a\overline{b}z}$,  $Y(z)=de^{-\overline{a}bz}$  and $(|c|^2+|d|^2) e^{|b|^2}=1$. That is,  the displays \eqref{xi1} and \eqref{abcd} are true.

\emph{Sufficiency.}  Assume $\xi_I(z)=ce^{-a\overline{b}z}+de^{-\overline{a}bz} J, \;\;\eta_I(z)=za+b,$ satisfying \eqref{abcd}, we need to show $A_{\xi,\eta} $ is isometric on $\mf$. For any $f\in \mf,$ the conditions yield the display \eqref{xietab} holds.  Lemma \ref{lem Ab} ensures that there exists some $g\in\mf$ such that $$g_I(z)=A_{-\overline{b}}^{-1} f_I(z)$$  and $\|g\|=\|f\|$, where $A_{-\overline{b}}^{-1} $ is the inverse mapping of $A_{-\overline{b}}.$  It follows that $A_{\xi,\eta}=W_{u,v} A_{-\overline{b}}^{-1}.$ Hence we obtain
\begin{eqnarray} \|A_{\xi,\eta} f\|=\|W_{u,v} A_{-\overline{b}}^{-1} f\|=\|W_{u,v} g\|=\|g\|=\|f\|, \end{eqnarray}
 which implies $A_{\xi,\eta} $ is isometric on $\mf$. This ends the proof.\zb

Now we proceed to investigate the equivalent forms for an involution of $A_{\xi,\eta}$ on $\mf.$ Comparing the blow results with Proposition \ref{prop iso}, it implies an involutive $A_{\xi,\eta}$ on $\mf$ must  be an isometry.

\begin{prop} \label{prop invo} The right-anti-linear  weighted composition  operator $A_{\xi,\eta} $ is involutive on $\mf$ if and only if the functions $\xi$ and $\eta$ are of the following forms,
\begin{eqnarray}  \xi_I(z)&=&ce^{-a\overline{b}z}+de^{-\overline{a}bz} J=(c+dJ)\star e^{-a\overline{b}z}, \nonumber\\ \eta_I(z)&=&za+b, \label{xi11}\end{eqnarray} where $a, b, c, d \in \mc$ satisfying \begin{eqnarray} |a|=1,\;\;(|c|^2+|d|^2) e^{|b|^2}=1\;\;\mbox{and}\;\;\overline{a}b+\overline{b}=0.\label{abcd+}\end{eqnarray}\end{prop}
\pf \emph{Necessity}. Since $ A_{\xi,\eta}:\mf\rightarrow \mf$ is an involutive, then $\eta(p)=pa+b$ with $0<|a|\leq 1$ and $a, b\in \mc.$  Furthermore, it holds $$A_{\xi,\eta}A_{\xi,\eta} f=f$$ for all $f\in \mf.$ That is to say
\begin{eqnarray} \xi_I(z)\star (\mathcal{C}\xi_I)( \eta_I(z))\star (\mathcal{C}[(\mathcal{C}f_I)\circ \eta_I])(\eta_I(z)) =f_I(z)\;\;\label{AXI}\end{eqnarray} $\mbox{for any}\; f\in \mf.$ In particular, letting $f_I(z)=1,$ we obtain that
\begin{eqnarray} \xi_I(z)\star (\mathcal{C}\xi_I)(\eta_I(z))=1.\label{xi} \end{eqnarray}
The above inequality reduces \eqref{AXI} into\begin{eqnarray} \mathcal{C}[(\mathcal{C}f_I)\circ \eta_I])(\eta_I(z)) =f_I(z)\label{ff}\end{eqnarray} for any $f\in \mf.$ Now let $f_I(z)=z$ in \eqref{ff}, we obtain that
\begin{eqnarray*} (\mathcal{C}\eta_I)(\eta_I(z))=\overline{a(\overline{az+b})+b}=z,\end{eqnarray*}
which entails that  $|a|^2=1$ and $\overline{a}b+\overline{b}=0.$  Since $|a|=1$ and the operator $ A_{\xi,\eta}:\;\mf\rightarrow \mf$ is bounded, Proposition \ref{prop lambda1} implies $\xi_I(z)=ce^{-a\overline{b}z}+d e^{-\overline{a}bz}J.$ Due to $-\overline{a}b=\overline{b}$ and $-a\overline{b}=b,$ it makes sense that
$\xi_I(z)=ce^{bz}+de^{\overline{b}z}J.$ And then \[\mathcal{C} \xi_I(z)=\overline{c}e^{\overline{b}z}+\overline{d}e^{bz}J.\] Moreover, it yields that  \begin{eqnarray} (\mathcal{C} \xi_I)(\eta_I(z))&=&\overline{c}e^{\overline{b}(az+b)}+\overline{d}e^{b(az+b)}J
\nonumber\\&=&\overline{c}e^{-bz+|b|^2}+\overline{d}e^{baz+b^2}J.\label{xieta}\end{eqnarray}
Combining \eqref{xi} with \eqref{xieta}, we deduce that
\begin{eqnarray*} 1&=&[ce^{bz}+d e^{\overline{b}z}J]\star [\overline{c}e^{-bz+|b|^2}+\overline{d}e^{baz+b^2}J]\nonumber\\&=& |c|^2e^{|b|^2}-d e^{\overline{b}z} d e^{\overline{ba}z+\overline{b}^2} +[ce^{bz}\overline{d}e^{baz+b^2}+d e^{\overline{b}z}ce^{-\overline{b}z+|b|^2}]J\nonumber\\&=&[ |c|^2e^{|b|^2}-d e^{\overline{b}z} d e^{\overline{ba}z+\overline{b}^2}] +[c\overline{d} e^{bz+baz+b^2}+dc e^{ |b|^2}]J,\end{eqnarray*}
which implies that
\begin{eqnarray} &&|c|^2e^{|b|^2}-d e^{\overline{b}z} d e^{\overline{ba}z+\overline{b}^2}=1,\label{=1}\\&&c\overline{d} e^{bz+baz+b^2}+dc e^{ |b|^2}=0.\label{=0}  \end{eqnarray}  In view of the fact $\eta$ is nonconstant  function, hence Proposition \ref{prop lambda1} ensures  at least one of $c$ and $d$  is not zero. We assume $c\neq 0,$ and then the display \eqref{=0} implies
\[ \overline{d} e^{bz+baz+b^2}=-d e^{ |b|^2}.\] Furthermore, it's also true that
\[  d  e^{\overline{b}\overline{z}+\overline{ba}\overline{z}+\overline{b}^2}=-\overline{d} e^{ |b|^2}.\] Now we let $w=\overline{z},$ it validates that
\begin{eqnarray}  d  e^{\overline{b}w+\overline{ba}w+\overline{b}^2}=-\overline{d} e^{ |b|^2}.\label{dd}\end{eqnarray} Putting \eqref{dd} into \eqref{=1}, it follows that
\[|c|^2e^{|b|^2}+|d|^2 e^{|b|^2}=1,\] which gives the condition in \eqref{abcd+}. On the other side, if $c=0,$ then $d$ must not zero. The display \eqref{=1} implies that \[d^2 e^{\overline{b}z+\overline{ba}z+\overline{b}^2}=-1.\] Thus $\overline{b}+\overline{ba}=0$ and $d^2 e^{\overline{b}^2}=-1.$  Combining $\overline{a}b=-\overline{b}$, it turns out that $\overline{a}(b-\overline{b})=0$ and  $b\in \mathbb{R}.$ Hence $d^2= -e^{-b^2}$ and \[(|c|^2+|d|^2)e^{|b|^2}=0+e^{-b^2} e^{|b|^2}=1.\] That is to say the conditions \eqref{xi11} and \eqref{abcd+} are satisfied.

\emph{ Sufficiency}. Suppose   the functions $\xi$ and $\eta$ own the forms in   \eqref{xi11} and \eqref{abcd+}. Then we deduce that \[ \xi_I(z)\star (\mathcal{C}\xi_I)(\eta_I(z))=1\] and  for $f_I(z)=U(z)+V(z)J,$ we conclude
\begin{eqnarray*}  (\mathcal{C}f_I)\circ \eta_I(z) = \overline{U\left(\overline{az+b}\right)}+\overline{V\left(\overline{az+b}\right)}J.\end{eqnarray*} Furthermore, employing  $|a|=1$ \;and\;$\overline{a}b+\overline{b}=0$, it  turns out that
\begin{eqnarray*} &&\mathcal{C}[(\mathcal{C}f_I)\circ \eta_I])(\eta_I(z)) \\&=& U\left(\overline{a\overline{\eta_I(z)}+b}\right)+ V\left(\overline{a\overline{\eta_I(z)}+b}\right)J\nonumber\\&=&
U\left(\overline{a}(az+b)+\overline{b}\right)+ V\left(\overline{a}(az+b)+\overline{b}\right)J
\nonumber\\&=&U(z)+V(z)J
\\&=&f_I(z).\end{eqnarray*} This ensures
 the formula \eqref{ff}, which shows $A_{\xi,\eta} $ is involutive on $\mf$. The proof is completed.
 \zb

 Particularly, \cite[Remark 3.10]{Le} indicates  there are infinitely many tuples $\{a, b, c, d=0\}$, for $|a|=|c|^2 e^{|b|^2}=1$ holds, but $a\overline{b}+b\neq 0$ is not satisfied. Thus combining   Proposition \ref{prop iso} with Proposition \ref{prop invo},  we can summarize a corollary as below.

\begin{cor} Suppose the right-anti-linear  weighted composition operator $ A_{\xi,\eta}:\;\mf\rightarrow \mf$ is bounded. If $A_{\xi,\eta} $ is involutive on $\mf$  then it is isometric. However,  an isometric operator  $A_{\xi,\eta} $  may not be involutive on $\mf$.  \end{cor}
 On account of Propositions  \ref{prop iso} and \ref{prop invo}   together with Definition \ref{conjugation}, we deduce a critical characterization for the operator $A_{\xi,\eta}$ is right-conjugation on $\mf$.

 \begin{thm} \label{thm conjugation}Let $\xi\in \mathcal{R}(\hh)$  and $\eta\in \mathcal{R}(\hh)$    nonconstant  satisfying $\eta (\mc)\subset \mc$ for some $I\in \ms.$  Suppose that $ A_{\xi,\eta}:\;\mf\rightarrow \mf$ is bounded, then $A_{\xi,\eta} $ is  a right-conjugation on $\mf$ if and only if the functions $\xi$ and $\eta$ are of the following forms,
 \begin{eqnarray*}  \xi_I(z)=(c+dJ)\star e^{-a\overline{b}z}=ce^{-a\overline{b}z}+de^{-\overline{a}bz} J, \;\;\eta_I(z)=za+b,  \end{eqnarray*} where $a, b, c \in \mc$ satisfying \begin{eqnarray*}|a|=1,\;\;(|c|^2+|d|^2) e^{|b|^2}=1\;\;\mbox{and}\;\;\overline{a}b+\overline{b}=0. \end{eqnarray*}\end{thm}

 Observing the condition $\overline{a}b+\overline{b}=0$ in Theorem \ref{thm conjugation}, we can rewrite the formula of $\xi$ as \[\xi_I(z)=ce^{bz}+de^{\overline{b}z}J.\]  Hence we can offer a  ultimate   characterization for our right-conjugation presented in Theorem \ref{thm conjugation}.The right-conjugation on $\mf$ will be denoted by $\mathcal{C}_{a,b,c,d}$ defined as \begin{eqnarray}(\mathcal{C}_{a,b,c,d} f_I)(z)=[ce^{bz}+de^{\overline{b}z}J]\star (\mathcal{C} f_I)(az+b),\label{cabcd}\end{eqnarray}where $a, b, c, d$ are constants satisfying \eqref{abcd+}.

  Particularly, we let $d=0$ in \eqref{cabcd} and the condition \eqref{abcd+}. We denote $\mathcal{C}_{a,b,c}=\mathcal{C}_{a,b,c,0}$.  The mapping
 $\mathcal{C}_{a,b,c}$ is a \emph{conjugation}. Immediately, we show there are many weighted composition operator commuting with $\mathcal{C}_{a,b,c}$.
 \begin{thm} \label{thm A} The weighted composition operator $W_{f,\varphi}$ is bounded on $\mf$ with noncostant $f_I(z)=F(z)+G(z)J$ and $\varphi(p)=pA+B$ satisfying $|A|\leq 1,\;A, B\in \mc$ for some $I\in \ms.$ Then $W_{f,\varphi}$ is  a  conjugate $\mathcal{C}_{a,b,c}-$commuting operator if and only if the following conditions hold:
 \begin{eqnarray} &&a(Az+B)+b=\overline{A}(az+b)+\overline{B}\nonumber\\&&
 \overline{a}(\overline{A}z+\overline{B})+\overline{b}=A(az+b)+B,\quad \quad \quad \label{AB}\\&&
 F(z)e^{b(Az+B)}=e^{bz}\overline{F(\overline{az+b})}\nonumber\\
 &&G(z)e^{\overline{b}(\overline{A}z+\overline{B})}=e^{bz}\overline{G(\overline{az+b})}
 \label{FGGF}.\end{eqnarray}
 \end{thm}
 \pf \emph{Sufficiency.} Suppose the conditions \eqref{AB} and \eqref{FGGF} hold and we will prove \[W_{f,\varphi} \mathcal{C}_{a,b,c} h_I= \mathcal{C}_{a,b,c}W_{f,\varphi} h_I\] for any $h\in \mf.$ Without loss of generality, we denote $h_I(z)=U(z)+V(z)J$. On the one hand, it yields that
 \begin{eqnarray*} \mathcal{C}_{a,b,c} h_I(z)&=& ce^{bz} \star \left(\overline{U(\overline{az+b})}+\overline{V(\overline{az+b})}J\right)
 \nonumber\\&=&ce^{bz}\overline{U(\overline{az+b})} +
 ce^{bz}\overline{V(\overline{az+b})} J,\end{eqnarray*}
 which entails
 \begin{eqnarray*} &&W_{f,\varphi} \mathcal{C}_{a,b,c} h_I(z)\\&=& [F(z)+G(z)J]\star  \mathcal{C}_{a,b,c} h_I(zA+B)
 \nonumber\\&=&
  [F(z)+G(z)J]\star \left[ce^{b(Az+B)} \overline{U(\overline{a(Az+B)+b})}\right.\nonumber\\&&\left.+ce^{b(Az+B)} \overline{V(\overline{a(Az+B)+b})}J\right]\nonumber\\
 &=&F(z)ce^{b(Az+B)} \overline{U(\overline{a(Az+B)+b})}-G(z)\overline{c}e^{\overline{b}(\overline{A}z+\overline{B})} V(\overline{a}(\overline{A}z+\overline{B})+\overline{b})\nonumber\\&&+ \left[F(z)ce^{b(Az+B)} \overline{V(\overline{a(Az+B)+b})}+G(z) \overline{c}e^{\overline{b}(\overline{A}z+\overline{B})} U(\overline{a}(\overline{A}z+\overline{B})+\overline{b})\right]J.\end{eqnarray*}
 On the other hand, we deduce that
 \begin{eqnarray*} &&W_{f,\varphi} h_I(z)\\&=&[F(z)+G(z)J]\star [U(Az+B)+V(Az+B)J]\nonumber\\&=&F(z)U(Az+B) -G(z)\overline{V(A\overline{z}+B)} \nonumber\\&&+
 \left[F(z)V(Az+B)+G(z)\overline{U(A\overline{z}+B)}\right]J .\end{eqnarray*}
 Furthermore, we formulate
  \begin{eqnarray*}&& \mathcal{C}_{a,b,c} W_{f,\varphi} h_I(z)\\&=&ce^{bz}  \overline{F(\overline{az+b})} \;\overline{U(A(\overline{az+b})+B)} \\&& -ce^{bz}\overline{G(\overline{az+b})} V(A (az+b)+B) \nonumber\\&&+ce^{bz}
 \left[\overline{F(\overline{az+b})}\; \overline{V(A(\overline{az+b})+B)}\right.\\&&\left.+
\overline{ G(\overline{az+b}) }U(A(az+b)+B)\right]J .\end{eqnarray*}
 Putting the conditions \eqref{AB} and \eqref{FGGF} into the above equation, we deduce that
  \begin{eqnarray*}&& \mathcal{C}_{a,b,c}W_{f,\varphi} h_I(z)\nonumber\\&=&cF(z)e^{b(Az+B)}\overline{U(\overline{a(Az+B)+b})}
  +\overline{c}G(z)e^{\overline{b}(\overline{A}z+\overline{B})}V(\overline{a}(\overline{A}z+\overline{B})+\overline{b})
\nonumber\\&&+ \left[c F(z)e^{b(Az+B)} \overline{V(\overline{a(Az+B)+b})}+\overline{c}G(z) e^{\overline{b}(\overline{A}z+\overline{B})} U(\overline{a}(\overline{A}z+\overline{B})+\overline{b})\right]J.  \end{eqnarray*} That is to say $W_{f,\varphi} \mathcal{C}_{a,b,c}h_I= \mathcal{C}_{a,b,c} W_{f,\varphi} h_I$.

\emph{ Necessity.} Since $W_{f,\varphi} \mathcal{C}_{a,b,c} h_I= \mathcal{C}_{a,b,c}W_{f,\varphi} h_I$ for any $h\in \mf.$ Particularly, we put $h_I(z)=1$ and $h_I(z)=z$, respectively in the above equation. Indeed, the case $U(z)=1$ and $V(z)=0$ entails \eqref{AB} and the case $U(z)=z$ and $V(z)=0$ ensures \eqref{FGGF}. This completes the proof.  \zb
\begin{rem} Let $a=c=1$ and $b=0$ in Theorem \ref{thm A} , which reduces to Theorem \ref{thm C-}.\end{rem}

\section{complex $\mathcal{C}_{a,b,c}-$symmetric weighted composition operators on $\mf$}

In this section, we study the weighted composition operators which are complex $\mathcal{C}_{a,b,c}-$symmetric on $\mf.$

 First, we analyze the action of the adjoint operator on the sequence of polynomials. Denoting \[K_{z}^{[m]}(q)=q^{m}K_{z}(q)=q^{m}\overline{K_{q}(z)},\] we have \[f^{[m]}(q)= \langle f, K_{q}^{[m]}\rangle.\]

\begin{lem} \label{lemm1}Suppose $W_{f,\varphi}$ is bounded in $\mf$, here $\varphi(p)= pa+b$, with $|a|\le 1, a,b\in \mathbb{C}_I$. Then
  \begin{eqnarray*} W_{f,\varphi}^{*}P_{m}=\sum_{j=0}^{m}\binom{m}{j}K_{b}^{[j]}\overline{a^{j}}\;\overline{f^{[m-j]}(0)}.
    \end{eqnarray*}\end{lem}
  \begin{proof}  For any $g\in \mf$, we have
   \begin{eqnarray*} \langle g, W_{f,\varphi}^{*}P_{m} \rangle &=&\langle W_{f,\varphi} g, P_{m} \rangle
   =\int_{\mathbb{C}_I} \overline{z^{m}}(W_{f,\varphi} g)_I(z)d\mu(z)\\&=&\int_{\mathbb{C}_I}\overline{K_{q}(z)}\overline{z^{m}}(W_{f,\varphi} g)_I(z)d\mu(z)|_{q=0}\\&=&W_{f,\varphi} g^{[m]}(0)=\left(f\star(g\circ \varphi)\right)^{[m]}|_{q=0}\\
   &=&\left(f\cdot(g\circ \varphi)\right)^{[m]}|_{q=0}\\
   &=&\sum_{j=0}^{m}\binom{m}{j}f^{[m-j]}(0)a^{j}g^{[j]}(b)
   \\&=&\sum_{j=0}^{m}\binom{m}{j}f^{[m-j]}(0)a^{j} \langle g, K_{b}^{[j]}\rangle
   \\&=& \left\langle g, \sum_{j=0}^{m}\binom{m}{j}K_{b}^{[j]}\overline{a^{j}}\;\overline{f^{[m-j]}(0)}\right\rangle   \end{eqnarray*}
the fifth equality was due to Remark \ref{rem R}.
\end{proof}

\begin{thm} \label{thml} Let $\mathcal{C}_{a,b,c}$ be a conjugation  on $\mf$ and $W_{f,\varphi}$  be  bounded weighted composition operator on $\mf$ with $f$ is not identically zero and $\varphi(p)= pA+B$, with $|A|\le 1, A, B\in \mathbb{C}_I$. Then the weighted composition operator $W_{f,\varphi}$ is complex $\mathcal{C}_{a,b,c}$-symmetric with respect to $P_{m}, \forall m\in \mathbb{N}$, i.e.
\[W_{f,\varphi}\mathcal{C}_{a,b,c}P_{m}= \mathcal{C}_{a,b,c}W_{f,\varphi}^{*}P_{m}\]
if and only if the function $f$ has the following form
\begin{eqnarray}f_I(z)=C_{1}e^{D_{1}z}+C_{2}e^{D_{2}z}J\label{f1}\end{eqnarray}
with  $C_{1}=F(0)$, $D_{1}=aB-bA+b$, $C_{2}=G(0)=-\frac{c^{2}}{|c|^{2}}\overline{C_2}e^{Bb-\overline{Bb}}$, $D_{2}=aB-\overline{bA}+b$ and $Im(Aa)=0.$
\end{thm}

\begin{proof} Suppose $W_{f,\varphi}$ is complex $\mathcal{C}_{a,b,c}$-symmetric with respect to $P_{m}, \forall m\in \mathbb{N}$. Taking $P_{0}(z)=1$, we have $\mathcal{C}_{a,b,c}P_{0}(z)=ce^{bz}$ and hence
 \begin{eqnarray*}
W_{f,\varphi}\mathcal{C}_{a,b,c}P_{0}(z)&=&(F(z)+G(z)J)\star ce^{b\varphi(z)}\\
&=&F(z)ce^{b\varphi(z)}+G(z)\overline{ce^{b\varphi(\bar{z})}}J.
 \end{eqnarray*}
 On the other hand, by Lemma \ref{lemm1}, we know \[W_{f,\varphi}^{*}P_{0}=K_{\varphi(0)}(z)\;\overline{F(0)+G(0)J},\] which gives
 \begin{eqnarray*}
 \mathcal{C}_{a,b,c}W_{f,\varphi}^{*}P_{0}(z)=ce^{bz}e^{\varphi(0)(az+b)}(F(0)-\overline{G(0)}J)
  \end{eqnarray*}
Therefore, the equality  \[W_{f,\varphi}\mathcal{C}_{a,b,c}P_{0}= \mathcal{C}_{a,b,c}W_{f,\varphi}^{*}P_{0}\] reduces to
   \begin{eqnarray*}
   F(z)&=&F(0)e^{(aB-bA+b) z}=C_{1}e^{D_{1}z},\\
   G(z)&=&-\frac{c^{2}}{|c|^{2}}\overline{G(0)}e^{(aB-\overline{bA}+b)z}e^{Bb-\overline{Bb}}=C_{2}e^{D_{2}z},
    \end{eqnarray*}
    where $C_{1}=F(0)$, $D_{1}=aB-bA+b$, $C_{2}=-\frac{c^{2}}{|c|^{2}}\overline{G(0)}e^{Bb-\overline{Bb}}$ and $D_{2}=aB-\overline{bA}+b.$ We also note that $F(0)$ and $G(0)$ are not all zero.

Furthermore suppose that the $f$ has the form alike \eqref{f1}, by Lemma \ref{lemm1}, we have
  \begin{eqnarray*}
  W_{f,\varphi}^{*}P_{m}(z)&=&\sum_{j=0}^{m}\binom{m}{j}z^{j}e^{\overline{B}z}\overline{A^{j}}
  \;\left(\overline{C_{1}D_{1}^{m-j}+C_{2}D^{m-j}_{2}J}\right).
      \end{eqnarray*}
Then
  \begin{eqnarray}
   && \mathcal{C}_{a,b,c}W_{f,\varphi}^{*}P_{m}(z)\nonumber\\&=&ce^{bz} \sum_{j=0}^{m}\binom{m}{j}(az+b)^{j}e^{B(az+b)}A^{j}\;\left(C_{1}D_{1}^{m-j}-\overline{C_{2}D^{m-j}_{2}}J\right)\nonumber\\
    &=&ce^{(b+aB)z+bB} \left(C_{1}[A(az+b)+D_{1}]^{m}-\overline{C_{2}}[A(az+b)+\overline{D_{2}}]^{m}J\right).\;\;\quad\quad\label{lh1}
    \end{eqnarray}
On the other hand, we have
  \begin{eqnarray}
   && W_{f,\varphi}\mathcal{C}_{a,b,c}P_{m}(z)\nonumber\\&=&  W_{f,\varphi}\left( ce^{bz}(az+b)^{m}\right) \nonumber\\
   &=&C_{1}ce^{(D_{1}+bA)z+bB}(a(Az+B)+b)^{m}\nonumber\\
   &&+C_{2}\overline{c}e^{(D_{2}+\overline{bA})z+\overline{bB}}(\overline{aA}z+\overline{aB+b})^{m}J
  \nonumber\\   &=&C_{1}ce^{(b+aB)z+bB}(A(az+b)+D_{1})^{m}\nonumber\\
   &&+C_{2}\overline{c}e^{(b+aB)z+\overline{bB}}(\overline{aA}z+Ab+\overline{D_{2}})^{m}J. \;\; \label{rh1}
     \end{eqnarray}

Now, comparing \eqref{lh1} and \eqref{rh1}, the equality
    \[ W_{f,\varphi}\mathcal{C}_{a,b,c}P_{m}(z)=\mathcal{C}_{a,b,c}W_{f,\varphi}^{*}P_{m}(z)\]
further requires \[-ce^{ bB} \overline{C_{2}}(A(az+b)+\overline{D_{2}})^{m}
=C_{2}\overline{c}e^{ \overline{bB}}(\overline{aA}z+Ab+\overline{D_{2}})^{m},\] which entails that
  \begin{eqnarray*}  \overline{C_{2}}(A az+Ab+\overline{D_{2}})^{m}=\overline{G(0)}(\overline{aA}z+Ab+\overline{D_{2}})^{m}
  \end{eqnarray*} for any $m\in \mathbb{N}.$ Particularly, letting $m=0$ and $m=1,$ it follows that \begin{eqnarray*} C_2=G(0)\;\;\mbox{and}\;\;Aa=\overline{Aa}.\end{eqnarray*} That further implies $C_{2}=-\frac{c^{2}}{|c|^{2}}\overline{C_2}e^{Bb-\overline{Bb}}$ and $Im(Aa)=0.$

   Conversely, based on the assumptions in \eqref{f1}, we can easily obtain the  $W_{f,\varphi}$ is complex $\mathcal{C}_{a,b,c}$-symmetric with respect to $P_{m}$ for all $m\in \mathbb{N}.$ This completes the proof.
\end{proof}
Further by Proposition \ref{prop spe}, we establish the conditions ensuring  the bounded weighted composition operators complex $\mathcal{C}_{a,b,c}-$symmetric on $\mf.$

\begin{thm}  Let $\mathcal{C}_{a,b,c}$ be a conjugation  on $\mf$, $f$ and $\varphi$ be two slice regular functions with $f$ not identically zero. Let further, $W_{f,\varphi}$  be  bounded  $\mf$ $($i.e.  $\varphi(p)= pA+B$, with $|A|\le 1, A, B\in \mathbb{C}_I)$. Then $W_{f,\varphi}$ is complex $\mathcal{C}_{a,b,c}$-symmetric  if and only if the following conditions hold:
\begin{eqnarray*}f_I(z)=C_{1}e^{D_{1}z}+C_{2}e^{D_{2}z}J \end{eqnarray*}
with  $C_{1}=F(0)$,  $C_{2}=G(0) $, $D_{1}=aB-bA+b$, $D_{2}=aB-\overline{bA}+b$,
$Im(Aa)=0$ and
\begin{eqnarray}\left\{
  \begin{array}{ll}
    \mbox{either}\; |A|<1, \\
    \mbox{or}\; |A|=1,\; D_1+A\overline{B}=0\;\mbox{and}\; D_2+\overline{A}B=0.
  \end{array}
\right.\label{ACD}\end{eqnarray}
\end{thm}
\pf \emph{Necessity}.  Theorem \ref{thml}  entails $f$ must be the form in \eqref{f1}.  Furthermore,
since $W_{f,\varphi}$ is bounded on $\mf$ and $f_I(z)=C_{1}e^{D_{1}z}+C_{2}e^{D_{2}z}J$, by Proposition \ref{prop spe}, we obtain \eqref{ACD}.

 \emph{Sufficiency}. Theorem  \ref{thml} together with  the facts  $P_{m}$ is a basis of $\mf$ and the weighted composition operator   $W_{f,\varphi}$ is bounded imply  \[ W_{f,\varphi}\mathcal{C}_{a,b,c} h(z)=\mathcal{C}_{a,b,c}W_{f,\varphi}^{*}h(z)\] for any $h\in \mf.$
 That is to say $W_{f,\varphi}$ is a complex $\mathcal{C}_{a,b,c}$-symmetric operator on $\mf$.
\zb

\section{appendix}

The quaternion exponential function $e_{\star}^{pq}=\sum_{n=0}^\infty \frac{p^n q^n}{n!}$ admits a closed  expression.

\begin{thm}\label{rem closed} For any $p=a+\omega b$, $q=c+\eta d$, where $\omega, \eta  $ belong to the 2-sphere $\mathbb{S}$. Then  function $e_{\star}^{pq}$   can be expressed as
\begin{eqnarray}e_{\star}^{pq}&=&\frac{1}{2}(\cos (bd-ac)+\omega \sin(bd-ac))e^{ac+bd}(1+\omega\eta)\nonumber\\&&+\frac{1}{2}(\cos (ac+bd)-\omega\sin (bd+ac))e^{ac-bd}(1-\omega\eta). \label{cc1}\end{eqnarray}  \end{thm}

\pf
Denoting \[p=a+\omega b=r_1(\cos x+ \omega \sin x)\] and \[q=c+\eta b=r_2(\cos y+\eta \sin y),\]  by direct computation, we have
\begin{eqnarray*}p^nq^n&=&r_1^nr_2^n(\cos nx+ \omega\sin nx)(\cos ny+\eta \sin ny)\nonumber\\&=&r_1^nr_2^n \cos nx \cos ny + r_1^nr_2^n \omega \sin nx \cos ny \\&& +r_1^nr_2^n \eta \cos nx \sin ny + r_1^nr_2^n \omega\eta\sin nx \sin ny.\end{eqnarray*}
Thus, the quaternion exponential can be split into four series,
\begin{align}e_{\star}^{pq}=K_1+K_2+K_3+K_4\label{pq}\end{align} where
\begin{eqnarray*}K_1:&=&\sum_{n=0}^\infty \frac{r_1^n r_2^n \cos nx \cos ny}{n!},\\ K_2:&=& \omega \sum_{n=0}^\infty \frac{r_1^n r_2^n \sin nx \cos ny}{n!},
\\ K_3:&=&\eta \sum_{n=0}^\infty \frac{r_1^n r_2^n \cos nx \sin ny}{n!},
\\ K_4:&=& \omega\eta\sum_{n=0}^\infty \frac{r_1^n r_2^n \sin  nx \sin ny}{n!}.\end{eqnarray*}

In the sequel, we   compute  the closed expression of each $K_i$, $i=1, 2, 3, 4$.
\begin{eqnarray*} K_1&=&\frac{1}{2} \sum_{n=0}^\infty \frac{(r_1r_2)^n(\cos n(x-y)+\cos n(x+y))}{n!}\nonumber\\&=&\frac{1}{4} \sum_{n=0}^\infty \frac{(r_1r_2)^n(e^{i n(x-y)}+e^{-i n(x-y)}+e^{i n(x+y)}+e^{-i n(x+y)}}{n!}\nonumber\\&=& \frac{1}{4}\left[e^{r_1r_2e^{i(x-y)}}+
e^{r_1r_2e^{-i(x-y)}}+e^{r_1r_2e^{i(x+y)}}+e^{r_1r_2e^{-i(x+y)}}\right]\\&=&
\frac{1}{2} [e^{r_1r_2\cos(x-y)}\cos (r_1r_2 \sin(x-y))\\&&+e^{r_1r_2\cos(x+y)}\cos (r_1r_2 \sin(x+y))].\end{eqnarray*}
Similarly, we obtain
\begin{eqnarray*} K_2&=& \frac{\omega}{2} [e^{r_1r_2\cos(x-y)}\sin (r_1r_2 \sin(x-y))\nonumber
\\&&+e^{r_1r_2\cos(x+y)}\sin (r_1r_2 \sin(x+y))]\nonumber\\
K_3&=&\frac{\eta}{2} [e^{r_1r_2\cos(x+y)}\sin (r_1r_2 \sin(x+y))\nonumber\\&&-e^{r_1r_2\cos(x-y)}\sin (r_1r_2 \sin(x-y))] \nonumber\\
K_4&=&\frac{\omega\eta}{2} [e^{r_1r_2\cos(x-y)}\cos (r_1r_2 \sin(x-y))\nonumber\\&&-e^{r_1r_2\cos(x+y)}\cos (r_1r_2 \sin(x+y))] \nonumber.\end{eqnarray*}
Putting the terms $K_i$, $i=1,2,3,4$ into \eqref{pq}, eliminating the angular variables and grouping corresponding terms, we obtain the desired closed expression \eqref{cc1}.
\zb

 \footnotesize{\bf Acknowledgments} The first author would like to express his gratitude to Zhenghua Xu (Hefei University of Technology) for many interesting discussions.  P. Lian is supported by the Tianjin Normal University starting grant No.004337. Y. X. Liang is supported by the National Natural Science Foundation of China (Grant No. 11701422).


\begin{thebibliography}{99}

\bibitem{ACSS} D. Alpay , F. Colombo, I. Sabadini, G. Salomon, The Fock space in the slice hyperholomorphic setting,   Springer, 2014, 43-59.


\bibitem{CSS1} F. Colombo,  I. Sabadini, D. C. Struppa, Slice monogenic functions,  noncommutative functional calculus. Springer Basel, 2011.

\bibitem{CSS2} F. Colombo,  I. Sabadini, D. C. Struppa,  Entire slice regular functions, Springer, 2016.

\bibitem{CM} C. C. Cowen and B. D. MacCluer, Composition operators on spaces of analytic functions,  CRC Press, Boca Raton, FL, 1995.

\bibitem{GP} S. R. Garcia, E. Prodan, M. Putinar, Mathematical and physical aspects of complex symmetric operators, J. Phys. A 47 (2014), 533-538.

\bibitem{GP1} S. R. Garcia and M. Putinar, Complex symmetric operators and applications. Trans. Amer. Math. Soc. 358(3)(2006),1285-1315.

\bibitem{GP2} S. R. Garcia and M. Putinar. Complex symmetric operators and applications. II. Trans. Amer. Math. Soc. 359(8)(2007), 3913-3931.

\bibitem{GP3} S. R. Garcia and M. Putinar, Interpolation and complex symmetry. Tohoku Math. J.  60  (2008), 423-440.

\bibitem{GW} S. Garcia, W. Wogen, Some new classes of complex symmetric operators, Trans. Amer. Math. Soc. 362(11) (2010), 6065-6077.


\bibitem{HK1} P. V. Hai, L. H. Khoi, Complex symmetry of weighted composition
operators on the Fock space, J. Math. Anal. Appl. 433(2016) 1757-1771.

\bibitem{HK2} P. V. Hai, L. H. Khoi, Weighted composition operators that are complex symmetric on the Fock space $\mathcal{F}^2(\mathbb{C}^n),$ C. R. Acad. Sci. Paris, Ser. I, 2016.

\bibitem{KM} S. Kumar, K. Manzoor, Slice regular Besov spaces of hyperholomorphic functions and composition operators, arXiv:1609.02394v1, 2016.

\bibitem{Le} T. Le, Normal and isometric weighted composition operators on the Fock space,  Bull. Lond. Math. Soc. 46(4)(2014), 847-856.

\bibitem{Ni}  N. K. Nikolski, Operators, functions, and systems: an easy reading: volume 2: model operators and systems, American mathematical Society, Providence R.I., 2002.

\bibitem{Sh} J. H. Shapiro, Composition operators and classical function theory, Springer-Verlag, New York, 1993. 

 \bibitem{VCG} C. Villalba, F. Colombo , J. Gantner, et al. Bloch, Besov and Dirichlet spaces of slice hyperholomorphic functions,  Complex Anal.   Oper. Theory,   9(2) (2014),1-39.

\bibitem{Zhu1} K. H. Zhu,  Analysis on Fock spaces, Springer, New York,  2012.

\end{thebibliography}
\end{document}